\documentclass[12pt]{article}
\usepackage{amsmath,amsfonts,amssymb,amsthm}
\usepackage{mathrsfs,hyperref,xcolor}
\usepackage[square,numbers]{natbib}

\newtheorem{proposition}{Proposition}[section]
\newtheorem{theorem}[proposition]{Theorem}
\newtheorem{corollary}[proposition]{Corollary}
\newtheorem{lemma}[proposition]{Lemma}
\newtheorem{remark}[proposition]{Remark}

\numberwithin{equation}{section}
\newcommand{\nc}{\newcommand}
\nc{\R}{{\mathbb R}}
\nc{\bS}{{\mathbb S}^{d-1}}
\nc{\N}{{\mathbb N}}
\nc{\Z}{{\mathbb Z}}
\nc{\BP}{\mathbb{P}}
\nc{\BE}{\mathbb{E}}
\nc{\BQ}{\mathbb{Q}}
\nc{\BM}{\mathbb{M}}
\nc{\bN}{{\mathbf N}}
\nc{\BX}{{\mathbb X}}
\nc{\BY}{{\mathbb Y}}
\nc{\cH}{{\mathcal H}}
\nc{\cB}{{\mathcal B}}
\nc{\cX}{{\mathcal X}}
\nc{\cN}{{\mathcal N}}
\nc{\cY}{{\mathcal Y}}
\nc{\dint}{{\rm d}}

\makeatletter
\def\keywords{\xdef\@thefnmark{}\@footnotetext}
\makeatother

\title{Criteria for Poisson process convergence with applications to inhomogeneous Poisson-Voronoi tessellations}
\author{

Federico Pianoforte\footnotemark[1] \, and \,  Matthias Schulte\footnotemark[2] }

\begin{document}
\footnotetext[1]{federico.pianoforte@stat.unibe.ch;  University of Bern, Institute of Mathematical Statistics and Actuarial Science.}
\footnotetext[2]{matthias.schulte@tuhh.de;  Hamburg University of Technology,  Institute of Mathematics.}
\maketitle

\begin{abstract}
This article employs the relation between probabilities of two consecutive values of a Poisson random variable to derive conditions for the weak convergence of point processes to a Poisson  process. As applications, we consider the starting points of $k$-runs in a sequence of Bernoulli random variables and point processes constructed using inradii and circumscribed radii of inhomogeneous Poisson-Voronoi tessellations.
\end{abstract}

\keywords{The research was supported by the Swiss National Science Foundation Grant No 200021\_175584}

\section{Introduction and main results}

Let $X$ be a random variable taking values in $\N_0=\N\cup\{0\}$ and let $\lambda>0$. It is well-known that 
\begin{equation}\label{eqm:characterizationPoisson}
k\mathbb{P}(X=k)= \lambda \mathbb{P}(X=k-1), \quad k\in\N,
\end{equation}
if and only if $X$ follows a Poisson distribution with parameter $\lambda$. We use this observation to establish weak convergence to a Poisson  process. Indeed, we will prove that a tight sequence of point processes $\xi_n$, $n\in \N$,  satisfies
$$
\lim_{n\rightarrow\infty} k\mathbb{P}(\xi_n(B)=k) - \lambda(B) \mathbb{P}(\xi_n(B)=k-1) =0, \quad k\in \N,
$$
for any $B$ in a certain family of sets and some locally finite measure $\lambda$, if and only if $\xi_n$ converges in distribution to a Poisson  process with intensity measure $\lambda$.  Many different methods to investigate Poisson  process convergence are available in the literature; we refer to surveys and classical results \cite{MR3642325,MR2933280, MR3992498}.
Using Stein's method, one can even derive quantitative bounds for the Poisson process approximation; see e.g.\ \cite{arratia1990,MR974580,MR1190904,MR1163825,ChenXia2004,MR2427851,MR3502603,otto2020poisson,MR2543874,Xia2005} and the references therein. In contrast to these results, our findings are purely qualitative and do not provide rates of convergence, but they have the advantage that the underlying conditions are easy to verify. This is demonstrated in Sections \ref{inradius} and \ref{circ-rad-proof}, where weak convergence of point processes constructed using inradii and circumscribed radii of inhomogeneous Poisson-Voronoi tessellations is established.

The proof of our abstract criterion for Poisson process convergence relies on characterizations of point process convergence from \cite{MR1876169,MR3642325} and the characterizing equation \eqref{eqm:characterizationPoisson} for the Poisson distribution.

Let us now give a precise formulation of our results. Let $S$ be a locally compact second countable Hausdorff space (lcscH space) with Borel $\sigma$-field $\mathcal{S}$. A non-empty class $\mathcal{U}$ of subsets of $S$ is called a ring if it is closed under finite unions and intersections, as well as under proper differences. Let $\widehat{\mathcal{S}}$ denote  the class of relatively compact sets of $S$.  We say that a measure $\lambda$ on $S$ is non-atomic if $\lambda(\{x\})=0$ for all $x\in S$, and we define
$$
\widehat{\mathcal{S}}_{\lambda} 
= \{ B\in\widehat{\mathcal{S}}\, : \, \lambda(\partial B) =0\} ,
$$
where $\partial B$ indicates the boundary of $B$. 

Let $\mathcal{M}(S)$ be the space of all locally finite measures on $S$, endowed with the vague topology induced by the mappings $\pi_f: \mu\mapsto \mu(f)=\int fd\mu$, $f\in C_K^+(S)$, where $ C_K^+(S)$ denotes the set of non-negative and continuous functions with compact support. Note that $\mathcal{M}(S)$ is a Polish space (see e.g.\ \cite[Theorem A2.3]{MR1876169}). Let $\mathcal{N}(S)\subset \mathcal{M}(S)$ denote the set of all locally finite counting measures.  
A random measure $\xi$ on $S$ is a random element in $\mathcal{M}(S)$ measurable with respect to the $\sigma$-field generated by the vague topology, and it is a point process if it takes values in $\mathcal{N}(S)$.

Our first main result provides a characterization of weak convergence to a Poisson  process.  
\begin{theorem}\label{TH}
	Let $\xi_n$, $n\in\mathbb{N}$, be a sequence of point processes, and let $\lambda$ be a non-atomic locally finite measure on $S$. Let  $\mathcal{U}\subseteq\widehat{\mathcal{S}}_{\lambda}$  be a  ring containing a countable topological basis of $S$. Then the following statements are equivalent:
	\begin{itemize}
		\item [(i)] For all open sets $B\in\mathcal{U}$ and $k\in\N$, $\xi_n(B)$, $n\in\N$, is tight and
		\begin{align}\label{eq:consecProbLim}
	    \lim_{n\rightarrow\infty}	k \mathbb{P}(\xi_n(B) =k) - \lambda(B) \mathbb{P}(\xi_n(B) = k-1)=0.
		\end{align}
		\item [(ii)] $\xi_n$, $n\in\N$, converges in distribution to a Poisson  process with intensity measure $\lambda$.
	\end{itemize}
\end{theorem}

\begin{remark}\label{rem:Tightness}
Note that the sequence $\xi_n(B)$, $n\in\N$, in Theorem \ref{TH} is tight by the Markov inequality if $\mathbb{E}[\xi_n(B)]\to \lambda(B)$.
\end{remark}

\begin{remark}
For a point process $\varrho$, the function $f:S\times\mathcal{N}(S)\rightarrow [0,\infty)$ defined as 
\begin{align}\label{test-funct}
f(x,\mu)
=\mathbf{1}_B(x)\mathbf{1}{\{\mu(B)=k\}}  
\end{align}
with $k\in\N$ and $B\in\mathcal{U}$ satisfies
\begin{equation}\label{eqn:MeckeLike}
 \mathbb{E}\sum_{x\in\varrho}f(x,\varrho) -\int_S\mathbb{E}\big[f(x,\varrho +\delta_x)\big]d\lambda(x)
= k \mathbb{P}(\varrho(B) =k) - \lambda(B) \mathbb{P}(\varrho (B) = k-1),
\end{equation}
where $\delta_{x}$ denotes the Dirac measure centered at $x\in S$. By the Mecke formula, the left-hand side of \eqref{eqn:MeckeLike} equals zero for all integrable functions $f:S\times\mathcal{N}(S)\rightarrow [0,\infty)$ if and only if $\varrho$ is a Poisson  process with intensity measure $\lambda$ (see e.g.\ \cite[Theorem 4.1]{MR3791470}).  Theorem \ref{TH} shows that, if we replace $\varrho$ by $\xi_n$, $n\in\N$, satisfying a tightness assumption, then the left-hand side of \eqref{eqn:MeckeLike} vanishes as $n\to\infty$ for all $f$ of the form \eqref{test-funct} if and only if $\xi_n$, $n\in\N$, converges weakly to a Poisson  process with intensity measure $\lambda$.  
\end{remark}

Next we apply Theorem \ref{TH} to investigate point processes on $S$ that are constructed from an underlying Poisson or binomial point process on a measurable space $(Y,\mathcal{Y})$. By $\mathcal{N}_\sigma(Y)$ we denote the set of all $\sigma$-finite counting measures on $Y$, which is equipped with the $\sigma$-field generated by the sets
\begin{align*}
\{
\mu\in  \mathcal{N}_\sigma(Y) \, : \, \mu(B)=k
\} , \quad k\in\N_0, B\in \mathcal{Y}.
\end{align*}
For $t\geq 1$ let $\eta_t$ be a Poisson  process on $Y$ with a $\sigma$-finite intensity measure $P_t$ (i.e.\ $\eta_t$ is a random element in $\mathcal{N}_\sigma(Y)$), while $\beta_n$ is a binomial point process of $n\in\mathbb{N}$ independent points in $Y$ which are distributed according to a probability measure $Q_n$. For a family of measurable functions $h_t: V_t\times\mathcal{N}_\sigma(Y)\rightarrow S$  with $V_t\in\mathcal{Y}$, $t\geq 1$, we are interested in the point processes
$$
\sum_{x \in \eta_{t}\cap V_t}\delta_{h_{t}(x,\eta_t)}, \quad t\geq 1, \quad \text{and} \quad \sum_{x\in\beta_n \cap V_n} \delta_{h_n(x,\beta_n)}, \quad n\in\N.
$$
In order to deal with both situations simultaneously, we introduce a joint notation. In the sequel, we study the point processes
\begin{equation}\label{generica}
\xi_{t}
=\sum_{x\in\zeta_{t}\cap U_t} \delta_{g_{ t}(x,\zeta_{ t})}, \quad t\geq 1,
\end{equation} 
where $\zeta_t=\eta_t$, $g_t=h_t$ and $U_t=V_t$ in the Poisson case, while $\zeta_t=\beta_{\lfloor t\rfloor}$, $g_t=h_{\lfloor t\rfloor}$ and $U_t=V_{\lfloor t\rfloor}$ in the binomial case. We assume
$$
\mathbb{P}(\xi_t(B)<\infty)=1 \quad  \text{ for all } B\in \widehat{\mathcal{S}}
$$
so that $\xi_t$ is  locally finite. Let $M_t$ be the intensity measure of $\xi_t$. By $K_t$ we denote the intensity measure of $\zeta_t$, i.e.\ $K_t=P_t$ if $\zeta_t=\eta_t$ and $K_t = \left\lfloor t \right\rfloor Q_{\lfloor t \rfloor}$ if $\zeta_t=\beta_{\lfloor t \rfloor}$. Moreover, we define $\hat{\zeta}_{t}=\eta_t$ in the Poisson case and $\hat{\zeta}_t=\beta_{\left\lfloor t \right\rfloor-1}$ in the binomial case. From Theorem \ref{TH} we derive the following criterion for convergence of $\xi_t$, $t\geq 1$, to a Poisson  process.

\begin{theorem}\label{real}
	Let $\xi_{t}, t\geq 1$, be a family of point processes on  $S$ given by \eqref{generica} and let $M$ be a non-atomic locally finite measure  on $S$. Fix any ring $\mathcal{U}\subset\widehat{\mathcal{S}}_{M}$ containing a countable topological basis, and assume that
	\begin{equation}\label{pplnew1}\begin{split}
	& \underset{t\rightarrow\infty}{\mathrm{lim}}\,M_t(B)
	=M (B)
	\end{split}\end{equation} 
	for all open sets $B\in\mathcal{U}$. 	Then,
	\begin{equation}\label{pplnew}\begin{split}
	 \underset{t\rightarrow\infty}{\lim} &\int_{U_t}\,\mathbb{E}\Big[\mathbf{1}\{g_t(x,\hat{\zeta}_{t}+\delta_x)\in B\}\mathbf{1}\Big\{\sum_{y\in\hat\zeta_{t}\cap U_t} \delta_{g_t(y,\hat\zeta_{t}+\delta_x)}(B)=m\Big\}\Big]dK_t(x)
	\\
	& - M(B)\mathbb{P}(\xi_t(B)=m)=0
	\end{split}\end{equation} 
	for all open sets $B\in\mathcal{U}$ and $m\in\mathbb{N}_0$, if and only if $\xi_t,t\geq 1,$ converges weakly to a Poisson  process with intensity measure $M$.
\end{theorem}

\begin{remark}\label{rem:Euclidean}
One is often interested in Poisson process convergence for   $S=\R^d$, $d\geq 1$, and for the situation that the intensity measure of the Poisson process is absolutely continuous (with respect to the Lebesgue measure). In this case, we can apply Theorem \ref{TH} and Theorem \ref{real} in the following way. The family $\mathcal{R}^d$ of sets in $\R^d$ that are finite unions of Cartesian products of bounded intervals is a ring contained in the relatively compact sets of $\R^d$. For any absolutely continuous measure the boundaries of sets from $\mathcal{R}^d$ have zero measure. By $\mathcal{I}^d$ we denote the subset of open sets of $\mathcal{R}^d$, which contains a countable topological basis of $\R^d$. Note that the sets of $\mathcal{I}^d$ are finite unions of Cartesian products of bounded open intervals.
 Thus,   we prove weak convergence for sequences of point processes on $\R^d$ to Poisson processes with absolutely continuous  locally finite intensity measures by showing \eqref{eq:consecProbLim} or \eqref{pplnew1} and \eqref{pplnew} for all sets from $\mathcal{I}^d$. For $d=1$ we use the convention $\mathcal{I}=\mathcal{I}^1$.
\end{remark}

Theorem \ref{real} says that in order to establish Poisson process convergence for point processes of the form \eqref{generica}, one has to deal with the dependence between
$$
\mathbf{1}\{g_t(x,\hat{\zeta}_t +\delta_x)\in B\}
\quad 
\text{ and }
\quad 
\mathbf{1}\Big\{\sum_{y\in\hat{\zeta}_t\cap U_t} \mathbf{1}\{g_t(y,\hat{\zeta}_t +\delta_x)\in B\}=k\Big\}.
$$
We say that a statistic is locally dependent if its value at a given point depends only  on a local and deterministic neighborhood. That is, for any fixed $x\in Y$, there exists a set $A_{t,x}\in\mathcal{Y}$ with  $x\in A_{t,x}$ such that 
\begin{equation}\label{eqn:local_dependence}
\mathbf{1}\{g_t(x,\hat{\zeta}_t +\delta_x)\in B\} = \mathbf{1}\{g_t(x,\hat{\zeta}_t|_{ A_{t,x}} +\delta_x)\in B\}.
\end{equation}
Here, we denote by $\mu|_A$ the restriction of a measure $\mu$ to a set $A$. For further notions of local dependence in the context of point processes  we refer to \cite{MR1190904,ChenXia2004,MR2427851}.
Next we describe heuristically how \eqref{eqn:local_dependence} can be applied to show \eqref{pplnew} in Theorem \ref{real} for $\zeta_t=\eta_t$ if
\begin{equation}\label{eqn:approximation}
\mathbf{1}\Big\{\sum_{y\in\eta_{t}\cap U_t} \delta_{g_t(y,\eta_{t})}(B)= m \Big\} \approx \mathbf{1}\Big\{\sum_{y\in\eta_{t}\cap{A_{t,x}^{c}}\cap U_t} \delta_{g_t(y,\eta_{t}|_{A_{t,x}^c})}(B)=m\Big\}
\end{equation}
for $x\in Y$. Under the assumption \eqref{eqn:local_dependence}, the integral in \eqref{pplnew} coincides with
$$
\int_{U_t} \mathbb{E}\Big[\mathbf{1}\{g_t(x,\eta_{t}|_{A_{t,x}}+\delta_x)\in B\}\mathbf{1}\Big\{\sum_{y\in\eta_{t}\cap U_t} \delta_{g_t(y,\eta_{t})}(B)= m \Big\}\Big]dK_t(x) .
$$
By \eqref{eqn:approximation}, the last expression can be approximated by
$$
\int_{U_t} \mathbb{E}\Big[\mathbf{1}\{g_t(x,\eta_{t}|_{A_{t,x}}+\delta_x)\in B\}\mathbf{1}\Big\{\sum_{y\in\eta_{t}\cap{A_{t,x}^{c}}\cap U_t} \delta_{g_t(y,\eta_{t}|_{A_{t,x}^c})}(B)=m\Big\}\Big]dK_t(x).
$$
Due to the independence of $\eta_{t}|_{A_{t,x}}$ and $\eta_{t}|_{A_{t,x}^c}$, this can be rewritten as
$$
\int_{U_t} \mathbb{P}\Big(\sum_{y\in\eta_{t}\cap{A_{t,x}^{c}}\cap U_t} \delta_{g_t(y,\eta_{t}|_{A_{t,x}^c})}(B)=m\Big)  \mathbb{E}\big[\mathbf{1}\{g_t(x, \eta_{t}|_{A_{t,x}} +\delta_x)\in B\}\big] dK_t(x).
$$
Using once more \eqref{eqn:local_dependence} and \eqref{eqn:approximation}, the previous term can be approximated by
$$
\mathbb{P}(\xi_t(B)=m)\int_{U_t} \mathbb{E}\big[\mathbf{1}\{g_t(x, \eta_{t} +\delta_x)\in B\}\big] dK_t(x) =  \mathbb{P}(\xi_t(B)= m) M_t(B),
$$
where the last equality follows from the Mecke formula. Consequently, the expression on the left-hand side of \eqref{pplnew} becomes small if the approximation in \eqref{eqn:approximation} is good. 

In Section \ref{sec:Applications}, we provide examples for applying our abstract main results Theorem \ref{TH} and Theorem \ref{real}. Our first example in Subsection \ref{succe} are $k$-runs, i.e.\ at least $k$ successes in a row in a sequence of Bernoulli random variables. For the situation that the success probabilities converge to zero, we show that the rescaled starting points of the $k$-runs behave like a Poisson  process if some independence assumptions on the underlying Bernoulli random variables are satisfied. 

As the second and third example, we consider statistics related to inradii and circumscribed radii of inhomogeneous Poisson-Voronoi tessellations. We study the Voronoi tessellation generated by a Poisson  process $\eta_t$, $t> 0$, on $\mathbb{R}^d$ with intensity measure $t\mu$, where $\mu$ is a locally finite and absolutely continuous measure with density $f$. In Section \ref{inradius}, for any cell with the nucleus in a compact set, we take the $\mu$-measure of the ball centered at the nucleus and with  twice the inradius as the radius. We prove that the point process formed by these statistics converges in distribution after a transformation depending on $t$ to a Poisson  process as $t\to\infty$ under some minor assumptions on the density $f$. Our transformation allows us to describe the behavior of the balls with large $\mu$-measures.  In Section \ref{circ-rad-proof}, we consider for each cell with the nucleus in a compact convex set the $\mu$-measure of the ball around the nucleus with the circumscribed radius as radius and establish, after rescaling with a power of $t$, convergence in distribution to a Poisson  process for $t\to\infty$. This result requires continuity of $f$, but under weaker assumptions on $f$, we provide  lower and upper bounds for the  tail distribution of the minimal $\mu$-measure of these balls having the circumscribed radii as radii.

In \cite{MR3252817}, the limiting distributions of the maximal inradius and the minimal circumscribed radius of a stationary Poisson-Voronoi tessellation were derived. In our work, we extend these results in two directions. First, our findings imply Poisson  process convergence of the transformed  inradii and  circumscribed radii for the stationary case.  This implies the mentioned results from \cite{MR3252817} and allows to deal with the $m$-th largest (or smallest) value or combinations of several order statistics. Second, we deal with inhomogeneous Poisson-Voronoi tessellations. In \cite{Chenavier2014} some general results for the extremes of stationary tessellations were deduced, but they cannot be applied to inhomogeneous Poisson-Voronoi tessellations. For stationary Poisson-Voronoi tessellations the convergence of the nuclei of extreme cells to a compound Poisson process was studied in \cite{ChenavierRobert2018}.

As our Theorem \ref{real} deals with underlying Poisson and binomial point processes, we expect that one can extend our results on inradii and circumscribed radii of Poisson-Voronoi tessellations to Voronoi tessellations constructed from an underlying binomial point process.

Before we discuss our applications in Section \ref{sec:Applications}, we prove our main results in the next section.

\section{Proofs of the main results}

Recall that $S$ is a locally compact second countable Hausdorff space, which is abbreviated as lcscH space. A topological space is second countable if its topology has a countable basis, and it is locally compact if every point has an open neighborhood whose topological closure is compact. A family of sets $\mathcal{C}\subset\widehat{\mathcal{S}}$ is called dissecting if
 \begin{itemize}
 	\item [(i)]  every open set $G\subset S$ can be written  as a countable union of sets in $\mathcal{C}$,
 	\item [(ii)] every relatively compact set $B\in\widehat{\mathcal{S}}$ is covered by finitely many sets in $\mathcal{C}$.
 \end{itemize}

\begin{lemma}\label{lem:dissect}
A countable topological basis $\mathcal{T}$ of $S$ is dissecting.
\end{lemma}

\begin{proof}
By the definition of a countable topological basis $\mathcal{T}$ has property (i) of a dissecting family of sets. Since, for any $B\in\widehat{\mathcal{S}}$, $\cup_{T\in\mathcal{T}}T=S \supset \overline{B}$, the compactness of $\overline{B}$ implies that (ii) is satisfied.
\end{proof}

Let us now state a consequence of \cite[Theorem 4.15]{MR3642325} and \cite[Theorem 16.16]{MR1876169} or \cite[Theorem 4.11]{MR3642325}. This result  will be used in the  proof of Theorem \ref{TH}. We write $\overset{d}{\to}$ to denote convergence in distribution.
\begin{lemma}\label{lem: equiv-Wconv} 
Let $\xi_n, n\in\N$, be a sequence of point processes on $S$, and let $\gamma$ be a Poisson process on $S$  with a non-atomic locally finite intensity  measure $\lambda$. Let $\mathcal{U}\subset \widehat{\mathcal{S}}_\lambda$ be a ring containing a countable topological basis. Then the following statements are equivalent:
\begin{itemize}
	\item [(i)]
	$\xi_n\xrightarrow{d} \gamma$.
	\item [(ii)]
	$\xi_n (B)\xrightarrow{d} \gamma(B)$ for all open sets $B\in\mathcal{U}$.
\end{itemize}
\end{lemma}
\begin{proof}
Observe that \cite[Theorem 3.6]{MR3791470} ensures the existence of a Poisson process $\gamma$ with  intensity measure $\lambda$. Since $\lambda$ has no atoms,  from \cite[Proposition 6.9]{MR3791470} it follows that $\gamma$ is a simple point process (i.e.\ $\mathbb{P}(\gamma(\{x\})\leq 1 \text{ for all } x\in S )=1$). Elementary arguments also yield $\widehat{\mathcal{S}}_\lambda=\{ B\in \widehat{\mathcal{S}}\, : \, \gamma(\partial B) =0\, \, \text{a.s.}\}=:\widehat{\mathcal{S}}_\gamma$.

It follows from Lemma \ref{lem:dissect} that $\mathcal{U}$ is dissecting. By \cite[Theorem 16.16 (ii)]{MR1876169} or \cite[Theorem 4.11]{MR3642325} with $\mathcal{U}$ as dissecting semi-ring, we obtain that $(i)$ implies $(ii)$.

Conversely, if $\xi_n(U)\xrightarrow{d} \xi(U)$ for all $U\in\mathcal{U}$, the desired result follows from \cite[Theorem 4.15]{MR3642325}, whose conditions are satisfied with $\mathcal{U}$ as dissecting ring and semi-ring. Thus, it is enough to show that $(ii)$ implies $\xi_n(U)\xrightarrow{d} \xi(U)$ for all $U\in\mathcal{U}$.

For any $U\in\mathcal{U}$ there exists a sequence of open sets $A_j,j\in\N$, such that  
$$
U\subset A_j, \quad A_{j+1}\subset A_j \quad \text{and}\quad  \overline{U}=\cap_{j\in\N}A_j.
$$
Since $\mathcal{U}$ contains a countable topological basis, for any $A_j$ one can find a countable family of open sets $B^{(j)}_\ell, \ell\in\N$, in $\mathcal{U}$ such that $\cup_{\ell\in\N}B^{(j)}_\ell = A_j$.  In particular, they cover the compact set $\overline{U}$. So there exists a finite subcover of elements from $B^{(j)}_\ell, \ell\in\N$, that covers $\overline{U}$. Since $\mathcal{U}$ is a ring, the union of the elements belonging to this subcover of $\overline{U}$ is in $\mathcal{U}$ for each $j\in\mathbb{N}$. Because $\mathcal{U}$ is closed under finite intersections, we can make this family of sets from $\mathcal{U}$ that contain $\overline{U}$ monotonously decreasing in $j$. Thus, without loss of generality, we may assume $ A_j\in\mathcal{U}$ for all $j\in\N$.

Since $\mathcal{U}$ is a ring and contains a countable topological basis, for the interior $\operatorname{int}(U)$ of $U$ there exists a sequence of open sets $B_j\in\mathcal{U}$, $j\in\N$, such that 
$$
 B_j\subset U, \quad B_{j}\subset B_{j+1} \quad \text{and} \quad  \operatorname{int}(U)=\cup_{j\in\N} B_j.
$$
For a fixed $m\in\N,$ we have that
\begin{align*}
	\mathbb{P}(\xi_n(B_j)\geq m) 
	 \leq \mathbb{P}(\xi_n (U)\geq m) \leq \mathbb{P}(\xi_n (A_j)\geq m) 
\end{align*}
for all $n\in\N$.  By $\xi_n(U')\xrightarrow{d}\gamma(U')$ for all open sets $U'\in\mathcal{U}$, we obtain
\begin{align}\label{eqn: ineqConvDistr}
	\mathbb{P}(\gamma(B_j)\geq m) 
	\leq \liminf_{n\to\infty}\mathbb{P}(\xi_n (U)\geq m) \leq \limsup_{n\to\infty}\mathbb{P}(\xi_n (U)\geq m)\leq \mathbb{P}(\gamma (A_j)\geq m).
\end{align}
Moreover, from $U\in\widehat{\mathcal{S}}_{\lambda}$, whence $\lambda(\partial U)=0$, it follows that $\lambda(B_j)\rightarrow \lambda(\operatorname{int}(U))= \lambda (U)$ and $\lambda(A_j)\to \lambda(\overline{U}) =\lambda(U)$ as $j\to\infty$. Thus, letting $j\to\infty$ in \eqref{eqn: ineqConvDistr} and using that $\gamma$ is a Poisson process lead to
$$
\lim_{n\rightarrow\infty}\mathbb{P}(\xi_n(U)\geq m)= \mathbb{P}(\gamma(U)\geq m).
$$
This establishes $\xi_n(U)\xrightarrow{d}\xi(U)$ and concludes the proof.
\end{proof}
We are now in the position to prove the first main result of this paper.
\begin{proof}[Proof of Theorem \ref{TH}] 
Let us show $(i)$ implies $(ii)$. By Lemma  \ref{lem: equiv-Wconv} it  is enough to prove that $\xi_n (B)\xrightarrow{d} \gamma(B)$ for all open sets $B\in\mathcal{U}$.  Since $\mathbb{P}(\xi_n(B)=0)$, $n\in\mathbb{N}$, is a bounded sequence in $[0,1]$, there exists a subsequence such that $\underset{j\rightarrow\infty}{\lim}\,\mathbb{P}(\xi_{n_j}(B)=0)$ exists; then repeated applications of \eqref{eq:consecProbLim} yield for $k\in\N$ that  
\begin{align}\label{a3}
\underset{j\rightarrow\infty}{\lim}\,\mathbb{P}(\xi_{n_j}(B)=k)
=\frac{\lambda(B)^k}{k!}\underset{j\rightarrow\infty}{\lim}\,\mathbb{P}(\xi_{n_j}(B)=0).
\end{align}
Consequently we have for any $N\in\mathbb{N}$,
\begin{align*}
\sum_{k=0}^N \underset{j\rightarrow\infty}{\lim}\,\mathbb{P}(\xi_{n_j}(B)=k)
& =\underset{j\rightarrow\infty}{\lim}\,\mathbb{P}(\xi_{n_j}(B)\in\{0,\dots,N\})
\\
& = 1-\underset{j\rightarrow\infty}{\lim}\,\mathbb{P}(\xi_{n_j}(B)\in\{N+1,N+2,\dots\}).
\end{align*} 
By tightness of $\xi_{n_j}(B)$, $j\in\mathbb{N}$, the right-hand side of the equation converges to 1 as $N\rightarrow\infty$ so that
$$
\sum_{k\in\mathbb{N}_0}\underset{j\rightarrow\infty}{\lim}\,\mathbb{P}(\xi_{n_j}(B)=k)=1.
$$
Thus, from  \eqref{a3} we deduce   $\underset{j\rightarrow\infty}{\lim}\,\mathbb{P}(\xi_{n_j}(B)=0)=e^{-\lambda(B)}$. Together with \eqref{a3},  this proves that
\begin{equation*}
\underset{j\rightarrow\infty}{\lim}\,\mathbb{P}(\xi_{n_j}(B)=k)=\frac{\lambda(B)}{k!}\,e^{-\lambda(B)}
\end{equation*}
for all $k\in\N_0$. In conclusion, since for any subsequence $(n_\ell)_{\ell\in\mathbb{N}}$ there exists a further subsequence $(n_{\ell_i})_{i\in\mathbb{N}}$ such that $\mathbb{P}(\xi_{n_{\ell_i}}(B)=0)\}$, $i\in\mathbb{N}$, converges to $e^{-\lambda(B)}$, we obtain $$
\underset{n\rightarrow\infty}{\lim}\,\mathbb{P}(\xi_{n}(B)=k)=\frac{\lambda(B)}{k!}\,e^{-\lambda(B)}
$$ 
for all $k\in\N_0$. The result follows by applying Lemma \ref{lem: equiv-Wconv}.

Conversely, let us assume $\xi_n\xrightarrow{d}\gamma$ for some Poisson  process $\gamma$ with intensity measure $\lambda$. It follows from Lemma \ref{lem: equiv-Wconv} that, for any open set $B\in\mathcal{U}$, $\xi_n(B)\xrightarrow{d}\gamma(B)$ so that $\xi_n(B), n\in\N$, is tight and
\begin{align*}
0
& = k\,\mathbb{P}(\gamma(B)=k) - \lambda(B)\mathbb{P}(\gamma(B)=k-1)
\\
&
= \lim_{n\rightarrow\infty} k\,\mathbb{P}(\xi_n(B)=k) - \lambda(B)\mathbb{P}(\xi_n(B)=k-1)
\end{align*}
for $k\in\N$, which shows $(i)$. 
\end{proof}
Finally, we derive Theorem \ref{real}  from Theorem \ref{TH}.
\begin{proof}[Proof of Theorem \ref{real}]
By \eqref{pplnew1} and the Markov inequality we deduce that $\xi_t(B), t\geq 1,$ is tight for all open $B\in\mathcal{U}$.   Let $f:S\times\mathcal{N}(S)\rightarrow [0,\infty)$ be the function given by
	\begin{align*}
	f(x,\mu)
	=\mathbf{1}_B(x)\mathbf{1}{\{\mu(B)=k\}}  
	\end{align*}
	for  $k\in\N$ and $B\in\mathcal{U}$. Then, by applying the Mecke equation (if $\zeta_t =\eta_t$) and the identity
	$$
	\BE\sum_{x\in \beta_n} u(x,\beta_n) = n \int_Y \BE[u(x,\beta_{n-1}+\delta_x)] dQ_n(x)
	$$
for any measurable function $u: Y\times \mathcal{N}_{\sigma}(Y)\to[0,\infty)$	(if $\zeta_t=\beta_{\lfloor t \rfloor}$), we obtain
\begin{align*}
& k\mathbb{P}(\xi_t(B)=k) =\mathbb{E}\sum_{z\in\xi_t}f(z,\xi_t) = \mathbb{E}\sum_{x\in \zeta_t\cap U_t } f(g_t(x,\zeta_t),\xi_t(\zeta_t))
\\
& = \int_{U_t}\,\mathbb{E}\Big[\mathbf{1}\{g_t(x,\hat{\zeta}_{t}+\delta_x)\in B\}\mathbf{1}\Big\{\sum_{y\in\hat\zeta_{t}\cap U_t} \delta_{g_t(y,\hat\zeta_{t}+\delta_x)}(B)=k-1\Big\}\Big]d K_t (x).
\end{align*}	
Thus, Theorem \ref{TH} yields the equivalence between \eqref{pplnew} and the convergence in distribution of $\xi_t, t\geq 1$, to a Poisson  process with intensity measure $M$.
\end{proof}

\section{Applications}\label{sec:Applications}

All our examples throughout this section concern point processes on $\R$. By Remark \ref{rem:Euclidean}, it is sufficient for the convergence of such point processes to a Poisson process on $\R$ with absolutely continuous locally finite intensity measure to show \eqref{eq:consecProbLim} or \eqref{pplnew1} and \eqref{pplnew} for all sets from $\mathcal{I}$, i.e.\ for all finite unions of open and bounded intervals.

\subsection{Long head runs}\label{succe}  
Consider a sequence of Bernoulli random variables. A $k$-head run is defined as an uninterrupted sequence of $k$ successes, where $k$ is a positive integer. For example, for $k = 1$, one simply studies the successes, while for $k = 2$ one considers the occurrence of two consecutive successes in a row. Several authors have investigated the number of $k$-head runs in a sequence of Bernoulli random variables; for an overview on this topic, we refer to \cite{MR1882476}. Let the starting point of a $k$-head run be the index of its first success. Our goal is to find explicit conditions under which the point process of rescaled starting points of the $k$-head runs converges weakly to a Poisson process. Our investigation relies on two assumptions: the probability of having a $k$-head run is the same  for all $k$ consecutive elements of the sequence, and  the Bernoulli random variables are independent if \emph{far  away}. We will see that if these conditions are satisfied and if the probability of having a $k$-head run goes to $0$ slower than the probability of having  a $k$-head run with at least another $k$-head run \emph{nearby}, then the aforementioned point process converges in distribution to a Poisson process.

Let us now give a precise formulation of our result. Let  $X^{(n)}_{i}$, $i,n\in\mathbb{N},$ be an array of Bernoulli distributed random variables and let $k\in\N$. Assume that there exists a function $f: \N \rightarrow \N$ such that for all $q,n\in\N$ the random variable $X^{(n)}_q$ is independent of $\{X^{(n)}_\ell : \vert q-\ell \vert \geq f(n), \, \ell\in\N\}$ and that
$$
y_n := \mathbb{P}\big( X^{(n)}_q=1,\dots,X^{(n)}_{q+k-1}=1\big)>0
$$
does not depend on $q$. If $X^{(n)}_{i}$, $i\in\mathbb{N}$, are i.i.d.\ for $n\in\N$, then $y_n=p_n^k$ with $p_n:=\mathbb{P}\big( X_1^{(n)}=1 \big)$. Define
\begin{align*}
I^{(n)}_i
=\mathbf{1}\big\{ X^{(n)}_i=1,\dots,X^{(n)}_{i+k-1}=1\big\}, \quad i\in\N.
\end{align*}
Let  $\xi_n$ be the point process of the $k$-head runs for $X_i^{(n)}, i\in\N$,  that is
\begin{align}\label{cons-succ}
\xi_n = \sum_{i=1}^\infty I_i^{(n)}\delta_{i y_n}. 
\end{align}
For any $i_0\in\N$, let
$$
W^{(n)}_{i_0}  
= \sum_{j\in \N \, : \,1\leq  \vert j-i_0 \vert \leq f(n)+k-2} I^{(n)}_j .
$$
We denote by $\lambda_1$ the restriction of the Lebesgue measure to $[0,\infty)$.
 
\begin{theorem}\label{Thm-S}
Let $\xi_n$, $n\in\N$,  be the sequence of  point processes given by \eqref{cons-succ}. Assume that $f(n)y_n\rightarrow 0$ and that
\begin{align}\label{eq:condition-on-sum}
\lim_{n\rightarrow\infty} \sup_{i\in \N} \, y_n^{-1} \mathbb{E}\big[I_i\mathbf{1}\{W_{i}^{(n)}>0\}\big] =0.
\end{align}
Then $\xi_n$ converges weakly to a Poisson  process with intensity measure $\lambda_1$.
\end{theorem}
For underlying independent Bernoulli random variables, the Poisson approximation of the random variable $\xi_n((0,u))$, $u>0$, is considered in e.g.\ \cite{arratia1990,MR1163825,MR1133732, MR1334894} and the Poisson process convergence follows from the results of \cite{arratia1990}. Quantitative bounds for the Poisson process approximation of $2$-runs in the i.i.d.\ case were derived in \cite[Proposition 3.C]{MR2543874} and \cite[Theorem 6.3]{Xia2005}; see also \cite[Subsection 3.5]{ChenXia2004}, where the Poisson process approximation for the more general problem of counting rare words is considered.

As a consequence of Theorem \ref{Thm-S}, we can study the limiting distribution of 
$$
T_n = \min \{i\in\N \, : \, I^{(n)}_i=1 \},
$$
which gives the first arrival time of a $k$-head run for a sequence of  Bernoulli random variables.

\begin{corollary} 
If the assumptions of Theorem \ref{Thm-S} are satisfied, then $y_nT_n$ converges in distribution to an exponentially distributed random variable with parameter $1$.
\end{corollary}

Clearly, in the case when the Bernoulli random variables $(X_i^{(n)})_{i\in\N} $ are i.i.d.\ with parameter $p_n>0$,  if $p_n$ converges to $0$, the assumptions of  Theorem \ref{Thm-S} are fulfilled with $f(n)\equiv 1$, and so $\xi_n$ converges in distribution to a Poisson  process. Other conditions for weak convergence are given in the following corollary.

\begin{corollary}\label{Cor:strong-cond-run-conv}
	Let $\xi_n$, $n\in\N$,  be the sequence of  point processes given by \eqref{cons-succ}.  Let us assume that $f(n)y_n\rightarrow 0$ and
	\begin{align*}
	\lim_{n\rightarrow\infty} \sup_{i\in \N} \, y_n^{-1} \sum_{j\in\N\, :\,  1\leq \vert i-j\vert \leq f(n) +k-2} \mathbb{E}[I^{(n)}_i I^{(n)}_j]=0 .
	\end{align*}
	Then $\xi_n$ converges weakly to a Poisson  process with intensity measure $\lambda_1$.
\end{corollary}

Let us now prove the main result of this section, Theorem \ref{Thm-S}.

\begin{proof}[Proof of Theorem \ref{Thm-S}]
 For any bounded interval $A\subset [0,\infty)$, the assumptions on $X_i^{(n)}, i\in\N$, imply that
$$
\mathbb{E}[\xi_n(A)]=  y_n \sum_{i=1}^\infty \delta_{y_n i}(A)  = 
(\sup(A) y^{-1}_n+ b_n )y_n- (\inf(A) y^{-1}_n+ a_n  )y_n
$$
for some $a_{n}, b_{n}\in [-1,1]$. By $y_n\to0$, we have $\mathbb{E}[\xi_n(A)]\to\lambda_1(A)$ and, consequently, $\mathbb{E}[\xi_n(B)]\to\lambda_1(B)$ for all $B\in\mathcal{I}$.  Moreover, $\xi_n(B)$, $n\in\N$, is tight (see Remark \ref{rem:Tightness}).
Then, we can write $\xi_n(B)$ as
$$
\xi_n(B)= \sum_{i\in\mathcal{A}_n} I^{(n)}_i
$$
with $\mathcal{A}_n=\{i\in\N: y_n i\in B\}$.

For $i_0\in \mathcal{A}_n$, we have for any $m\in\N$ that 
\begin{align*}
& \big\vert \mathbb{E}\big[ I^{(n)}_{i_0} \mathbf{1}\big\{\xi_n(B)- I^{(n)}_{i_0}=m-1\big\}\big] 
- \mathbb{E}\big[I^{(n)}_{i_0} \mathbf{1}\big\{\xi_n(B) - W^{(n)}_{i_0} - I^{(n)}_{i_0}=m-1\big\}\big] \big\vert  \\
&\leq \mathbb{E}\big[I^{(n)}_{i_0}\mathbf{1}\{ W^{(n)}_{i_0}>0\}\big].
\end{align*}
Together with $\mathbb{E}\big[\xi_n(B)\big]=|\mathcal{A}_n| y_n$, this yields
\begin{align*}
H_n &:=\Big\vert\sum_{i \in \mathcal{A}_n } \mathbb{E}\big[ I^{(n)}_{i} \mathbf{1}\big\{\xi_n(B)- I^{(n)}_{i}=m-1\big\}\big] 
- \sum_{i \in \mathcal{A}_n }\mathbb{E}\big[I^{(n)}_{i} \mathbf{1}\big\{\xi_n(B) - W^{(n)}_{i} - I^{(n)}_{i}=m-1\big\}\big] \Big\vert \\
& \leq 
\sum_{i\in\mathcal{A}_n} \mathbb{E}\big[I^{(n)}_{i}\mathbf{1}\{ W^{(n)}_{i}>0\}\big] \leq \Big( \sup_{i\in \N} \, y_n^{-1} \mathbb{E}\big[I_i\mathbf{1}\{ W^{(n)}_{i}>0\}\big]\Big)\mathbb{E}[\xi_n(B)].
\end{align*}
Therefore from  \eqref{eq:condition-on-sum}, we obtain $H_n\to0$. From the independence of $I^{(n)}_{i_0}$ and $\xi_n(B) - W^{(n)}_{i_0} - I^{(n)}_{i_0}$ for $i_0\in\mathcal{A}_n$, it follows that 
$$
\mathbb{E}\big[ I^{(n)}_{i_0} \mathbf{1}\big\{\xi_n(B)- W^{(n)}_{i_0} - I^{(n)}_{i_0}=m-1\big\}\big] 
=\mathbb{E} \big[ I^{(n)}_{i_0} \big]\mathbb{P}\big(\xi_n(B) - W^{(n)}_{i_0} - I^{(n)}_{i_0}=m-1\big).
$$
Combining the previous arguments implies for $m\in\N$ that
\begin{align*}
&\limsup_{n\rightarrow\infty}\big\vert m\mathbb{P}(\xi_n(B) = m ) - \lambda_1(B)\mathbb{P}(\xi_n (B)
 = m-1) \big\vert 
 \\
 &=\limsup_{n\rightarrow\infty}\Big\vert \sum_{i \in \mathcal{A}_n } \mathbb{E}\big[ I^{(n)}_{i} \mathbf{1}\big\{\xi_n(B)- I^{(n)}_{i}=m-1\big\}\big]  - \lambda_1(B)\mathbb{P}(\xi_n (B)
 = m-1) \Big\vert 
 \\
 & = \limsup_{n\rightarrow\infty}\Big\vert \sum_{i \in \mathcal{A}_n }\mathbb{E}\big[I^{(n)}_{i} \mathbf{1}\big\{\xi_n(B) - W^{(n)}_i - I^{(n)}_i=m-1\big\}\big] - \mathbb{E}[\xi_n(B)]\mathbb{P}(\xi_n(B)=m-1) \Big\vert
\\
 & = \limsup_{n\rightarrow\infty}\Big\vert \sum_{i \in \mathcal{A}_n }\mathbb{E}[I^{(n)}_{i}]\mathbb{P}\big(\xi_n(B) -  W^{(n)}_i - I^{(n)}_i=m-1\big) -  \sum_{i\in\mathcal{A}_n} \mathbb{E}[I_i^{(n)}]  \mathbb{P}(\xi_n(B)=m-1) \Big\vert
\\
& \leq \limsup_{n\rightarrow\infty} \sum_{i \in \mathcal{A}_n } \mathbb{E}[I^{(n)}_i]\mathbb{P}\big( W^{(n)}_{i} + I^{(n)}_{i}>0 \big) \leq \lambda_1(B)  \limsup_{n\rightarrow\infty}  \sup_{i\in \N }\mathbb{P}\big(  W^{(n)}_{i} + I^{(n)}_{i}>0 \big).
\end{align*}
Finally, the inequality 
$$
\mathbb{P}\big(  W^{(n)}_{i} + I^{(n)}_{i}>0 \big)\leq   (2k+2f(n) - 3) y_n , \quad i\in\N,
$$
and the assumption  $f(n)y_n\rightarrow 0$ lead to  
$$
\lim_{n\rightarrow\infty}\big\vert m\mathbb{P}(\xi_n(B) = m ) - \lambda_1(B)\mathbb{P}(\xi_n (B)
= m-1) \big\vert  =0.
$$
The result follows by  applying Theorem \ref{TH}.
\end{proof}
\begin{proof}[Proof of Corollary \ref{Cor:strong-cond-run-conv}]
This follows directly from Theorem \ref{Thm-S} and
\begin{align*}
 \mathbb{E}\big[I_i \mathbf{1}\{W^{(n)}_{i}>0\} \big]  &\leq  \mathbb{E}\big[I^{(n)}_{i}W^{(n)}_{i}\big]= \sum_{ j \in \N \, :\, 1\leq  \vert i-j\vert \leq f(n) +k -2} \mathbb{E}\big[ I^{(n)}_i I^{(n)}_j\big]
\end{align*}
for any $i\in \N$.
\end{proof}

\subsection{Inradii of an inhomogeneous Poisson Voronoi tessellation}\label{inradius}
In this section, we consider  the  inradii of an inhomogeneous  Voronoi tessellation generated by a Poisson  process with a certain intensity measure $t\mu, t>0$;  recall that the inradius of a cell is  the largest radius for which the ball centered at the nucleus is contained in the cell. We study the point process on $\R$ constructed by taking  for any cell with the nucleus in a compact  set, a transform of the $\mu$-measure of the ball centered at the nucleus and with  twice the inradius as the radius. The aim is to continue the work started in \cite{MR3252817} by extending the  result on the largest inradius to  inhomogeneous  Poisson-Voronoi tessellations and  proving weak convergence of    the aforementioned  point process to a Poisson process. 

For any locally finite counting measure $\nu$ on $\R^d$, we denote by $N(x,\nu)$ the Voronoi cell with nucleus $x\in\R^d$ generated by $\nu + \delta_x$, that is
$$
N(x,\nu)
=\left\{y\in\mathbb{R}^d\,:\,\Vert x-y\Vert  \leq   \Vert  y-x'\Vert ,x\neq x'\in \nu\right\},
$$
where $\Vert \cdot \Vert$ denotes the Euclidean norm. For $x\in\nu$ we have $N(x,\nu)=N(x,\nu-\delta_x)$.  Voronoi tessellations, i.e.\ tessellations consisting of Voronoi cells $N(x,\nu)$, $x\in\nu$, arise in different fields such as biology \cite{POUPON2004233}, astrophysics \cite{refId0} and communication networks \cite{MR2150719}. For more details on Poisson-Voronoi tessellations, i.e.\ Voronoi tessellations generated by an underlying Poisson  process, we refer the reader to e.g.\  \cite{MR2654678,MR1295245,MR2455326}. The inradius of the Voronoi cell $N(x,\nu)$ is given by
$$
c(x,\nu )= \sup\{R\geq0 \ : \mathbf{B}(x,R)\subset N(x,\nu)\},
$$
where $\mathbf{B}(x,r)$ denotes the open ball centered at $x\in \R^d$ with radius $r>0$.

Let $\eta_{t}$, $t>0$, be a Poisson  process on $\R^d$ with intensity measure $t \mu$, where $\mu$ is a locally finite measure on $\R^d$ with density $f: \R^d \rightarrow [0,\infty)$. Consider a compact set $W\subset\mathbb{R}^d$ with $\mu(W)=1$, and assume that there exists a bounded open set  $A\subset \R^d$ with $W\subset A$ such that  $f_{min}:=\inf_{x\in A} f(x) >0$ and $f_{max}:=\sup_{x\in A} f(x) <\infty$. For any Voronoi cell $N(x,\eta_t)$ with $x\in\eta_t,$ we take the $\mu$-measure of the ball around $x$ with twice the inradius as radius, and we define the point process  $\xi_t$ on $\R$ as
\begin{equation}\label{max-irr-xi}\begin{split}
\xi_t 
& = \xi_t(\eta_t) =\sum_{x\in\eta_t\cap W} \delta_{ t \mu(\mathbf{B}(x,2c(x,\eta_t ))) - \log(t)} .
\end{split}\end{equation}
Let $M$ be  the measure on $\R$ given by $M([u,\infty))= e^{-u}$ for $u\in\R$.
\begin{theorem}\label{thm: max inradius}
	Let $\xi_t$, $t>0$, be the family of point processes on $\R$ given by \eqref{max-irr-xi}. Then $\xi_t$ converges in distribution to a Poisson  process with intensity measure $M$. 
\end{theorem}
The next theorem shows that if the density function $f$ is H\"older continuous,   it is possible to take out the factor $2$ from $\mu(\mathbf{B}(x,2c(x,\eta_t )))$ and to consider $2^d\mu(\mathbf{B}(x, c(x,\eta_t )))$. Recall that a function $h:\R^d\to\R$ is H\"older continuous with exponent $b>0$ if there exists a constant $C>0$ such that
$$
|h(x)-h(y)| \leq C \|x-y\|^b
$$
for all $x,y\in\R^d$.
We define the point process $\widehat{\xi}_t$ as 
$$
\widehat{\xi}_t =\widehat{\xi}_t(\eta_t) = \sum_{x\in\eta_t\cap W} \delta_{2^d t \mu (\mathbf{B}(x, c(x,\eta_t)))- \log (t)}.
$$
\begin{theorem}\label{Thm:inr-cont-dens}
Let $f$ be H\"older continuous. Then $\widehat{\xi}_t, t>0,$   converges in distribution to a Poisson  process with intensity measure $M$.
\end{theorem}
As a corollary of the previous theorems, we have the following generalization to the inhomogeneous case of the result obtained in \cite[Theorem 1, Equation (2a)]{MR3252817}  for the stationary case; see also \cite[Section 5]{Chenavier2014} for the maximal inradius of a stationary Poisson-Voronoi tessellation and of a stationary Gauss-Poisson-Voronoi tessellation.  
\begin{corollary}\label{cor:maxInr}
For $u\in\R$,
\begin{align}\label{eq:max-inr}
	\lim_{t\rightarrow\infty} \mathbb{P}\big( \max_{x\in\eta_t\cap W} t \mu(\mathbf{B}(x,2c(x,\eta_t ))) - \log(t) \leq u  \big) = e^{-e^{-u}}.
\end{align}
Moreover, if $f$ is H\"older continuous, 
\begin{align}\label{eq:max-inr-contf}
	\lim_{t\rightarrow\infty} \mathbb{P}\big( \max_{x\in\eta_t\cap W} 2^d t \mu(\mathbf{B}(x,c(x,\eta_t ))) - \log(t) \leq u  \big) = e^{-e^{-u}}.
\end{align}
\end{corollary}

For an underlying binomial point process, \eqref{eq:max-inr} was shown under similar assumptions in \cite{GyoerfiHenzeWalk2019}. The related problem of maximal weighted $r$-th nearest neighbor distances for the points of a binomial point process was studied in \cite{Henze1982}; see also \cite{Henze1983}.

For the proofs of Theorem \ref{thm: max inradius} and Theorem \ref{Thm:inr-cont-dens}, we will use the quantities $v_t(x,u)$ and $q_t(x,u)$, which are introduced in the next lemma.

\begin{lemma}\label{lm : exist and bound of v}
For any $u\in\R$ there exists a $t_0 >0$ such that for all $x\in W$ and $t>t_0$ the equations
\begin{align}\label{eq: def v}
 t \mu(\mathbf{B}(x, 2v_t(x,u)))= u + \log (t) \quad
\text{ and }
\quad
 2^dt \mu(\mathbf{B}(x, q_t(x,u)))= u + \log (t)
\end{align}	
have  unique solutions $v_t(x,u)$ and $q_t(x,u)$, respectively, which satisfy 
\begin{align}\label{eq: bnd v}
	\max\{v_t(x,u), q_t(x,u)\}\leq \Big(\frac{ u + \log (t)}{ 2^d f_{min} k_d t }\Big)^{1/d},
\end{align}
where $k_d$ is the volume of the $d$-dimensional unit ball.
\end{lemma}

\begin{proof}
Let $u\in\R$ be fixed and set $m = \inf\{ \Vert x-y \Vert \, : \, x\in \partial W, y \in\partial A\}\in (0,\infty)$. Note that $B(x,m)\subset A$ for all $x\in W$. Choose $t_0>0$ such that
$$
\frac{u + \log(t)}{t} < f_{min} k_d m^d
$$
for $t>t_0$. For $x\in W$ and $t>t_0$ this implies that
$$
2^dt \mu(\mathbf{B}(x, m))\geq t \mu(\mathbf{B}(x, m)) \geq t  f_{min} k_d m^d > u + \log(t)
$$
and, obviously, $t \mu(\mathbf{B}(x, 0))=0$. Since the function $[0,m]\ni a\rightarrow \mu(\mathbf{B}(x,a))$ is continuous and strictly increasing because of $f_{min}>0$, by the intermediate value theorem, the equations in \eqref{eq: def v} have unique solutions $v_t(x,u)$ and $q_t(x,u)$. Since $\max\{2v_t(x,u), q_t(x,u)\} < m$, we obtain
$$
\frac{u + \log (t) }{t}= \mu(\mathbf{B}(x, 2v_t(x,u))) \geq 2^d f_{min} k_d v_t(x,u)^d
$$
and
$$
\frac{u + \log (t) }{t}= 2^d\mu(\mathbf{B}(x, q_t(x,u))) \geq 2^d f_{min} k_d q_t(x,u)^d,
$$
which prove \eqref{eq: bnd v}.
\end{proof}

Let $M_t$ be  the intensity measure of $\xi_t$. Then from the Mecke formula and Lemma \ref{lm : exist and bound of v} it follows that for any $u\in\R$ there exists a $t_0>0$ such that for $t>t_0$,
\begin{align}
M_t([u,\infty)) 
& =  t \int_W \mathbb{P}\big(t \mu(\mathbf{B}(x,2c(x,\eta_t + \delta_x ))) \geq u+ \log(t)\big) f(x)dx 
\notag
\\
&= t \int_W \mathbb{P}(c(x,\eta_t + \delta_x) \geq v_t(x,u)) f(x) dx
\notag
\\
&=   t \int_W \mathbb{P}\big( \eta_t(\textbf{B}(x, 2v_t(x,u)))=0\big) f(x) dx
\notag
\\
& = t \int_W e^{-t\mu(\textbf{B}(x, 2v_t(x,u)))} f(x)dx = t e^{-u-\log(t)} \mu(W) 
= e^{-u} =M([u,\infty)) ,
\label{int-mea-xi}
\end{align}
 where we used \eqref{eq: def v} and $\mu(W)=1$ in the last steps. For any $y\in\R$ and point configuration $\nu$ on $\R^d$ with $y\in\nu$, we denote by $h_t(y, \nu)$ the quantity
\begin{align}\label{h_t-sec-1.3.3}
h_t(y, \nu) 
=   t \mu(\mathbf{B}(y,2c(y,\nu))) - \log(t),
\end{align}
where $c(y,\nu)$ is the inradius of the Voronoi cell with nucleus $y$ generated by $\nu$. So we can rewrite $\xi_t$ as
$$
\xi_t
=\sum_{x\in\eta_t\cap W} \delta_{h_t(x,\eta_t)}.
$$
\begin{proof}[Proof of Theorem \ref{thm: max inradius}]
	From  Theorem \ref{real} and  \eqref{int-mea-xi} it follows that it is enough to show that  
	\begin{equation}\label{prfthciin}\begin{split}
	& \underset{t\rightarrow\infty}{\lim} \,t \int_{W}\mathbb{E}\Big[\mathbf{1}\{h_t(x,\eta_t + \delta_x)\in B\}\mathbf{1}\Big\{\sum_{y\in\eta_{t}\cap W} \delta_{h_t(y,\eta_{t}+\delta_x)}(B)=m\Big\}\Big] f(x)dx
	\\
	& \qquad -M(B)\mathbb{P}(\xi_t(B)=m) = 0 
	\end{split}\end{equation} 
	for any $m\in\N_0$ and $B\in\mathcal{I}$. Let $B= \bigcup_{j=1}^{\ell} (u_{2j-1}, u_{2j})$ with $u_1< u_2 < \dots < u_{2\ell}$ and $\ell\in \N$. By Lemma \ref{lm : exist and bound of v} there is a $t_0>0$ such that $v_t(x,u_k)$ exists for all $k=1,\dots, 2\ell,$ $x\in W$ and $t>t_0$. Assume $t>t_0$ in the following. Elementary arguments imply that 
	\begin{align}
	& t \int_{W} \mathbb{E}\Big[ \mathbf{1}\{h_t(x,\eta_t + \delta_x)\in B\}\mathbf{1}\Big\{\sum_{y\in\eta_t\cap W} \delta_{h_t(y,\eta_t +\delta_x)}(B)=m \Big\} \Big] f(x)d x
	\notag
	\\
	& = \sum_{j=1}^\ell t \int_{W} \mathbb{E}\Big[ \mathbf{1}\{h_t(x,\eta_t + \delta_x)\in (u_{2j-1}, u_{2j})\}\mathbf{1}\Big\{\sum_{y\in\eta_t \cap W} \delta_{h_t(y,\eta_t +\delta_x)}(B)=m \Big\} \Big] f(x)d x.
	\label{sum-step0}
	\end{align}
    For each $k=1,\dots, 2\ell$, set $w_{k,t,x} = 2v_t(x,u_{k})$. Since $h_t(x,\eta_t + \delta_x)\in (u_{2j-1},u_{2j})$ if and only if $c(x,\eta_t + \delta_x) \in ( v_t(x,u_{2j-1}),v_t(x,u_{2j}) ),$ or equivalently, $\eta_t(\textbf{B}(x, w_{2j-1,t,x}))=0$ and $\eta_t(\textbf{B}(x, w_{2j,t,x}))>0$, we obtain that 
   	\begin{align*}
	S_j :& =  t \int_{W} \mathbb{E}\Big[ \mathbf{1}\{h_t(x,\eta_t + \delta_x)\in (u_{2j-1}, u_{2j})\}\mathbf{1}\Big\{\sum_{y\in\eta_t \cap W} \delta_{h_t(y,\eta_t +\delta_x)}(B)=m \Big\} \Big] f(x) d x
	\notag
	\\
	& = t \int_{W} \mathbb{E}\Big[\mathbf{1} \left\{ \eta_t\big(\textbf{B}(x, w_{2j-1,t,x})\big)=0\right\}
	\notag \\
	& \hspace{2.5cm} \times \mathbf{1}\Big\{\eta_t\big(\textbf{B}(x, w_{2j,t,x})\big)>0, \sum_{y\in\eta_t \cap W} \delta_{h_t(y,\eta_t +\delta_x)}(B)=m \Big\}\Big] f(x) d x.
	\end{align*}
	For any point configuration $\nu$ on $\R^d$ and $x\in W$, let $\xi_{t,x}(\nu)$ be the counting measure given by $\xi_{t,x}(\nu)=\sum_{y\in\nu \cap W} \delta_{h_t(y,\nu + \delta_x)}$ so that
	\begin{align*}
	S_j & =  t \int_{W} \mathbb{E}\big[\mathbf{1} \left\{ \eta_t\big(\textbf{B}(x, w_{2j-1,t,x})\big)=0\right\}
	\\
	& \hspace{2.5cm}\times\mathbf{1}\left\{\eta_t\big(\textbf{B}(x, w_{2j,t,x})\big)>0, \xi_{t,x}\big(\eta_t|_{\textbf{B}(x,w_{2j-1,t,x})^c} \big)(B)=m \right\}\big] f(x) d x 
	\\
	& = t \int_{W} \mathbb{P}\left(\eta_t\big(\textbf{B}(x, w_{2j-1,t,x})\big)=0\right)
	\\
	& \qquad \times \mathbb{P}\big(\eta_t\big(\textbf{B}(x, w_{2j,t,x})\setminus \textbf{B}(x, w_{2j-1,t,x})\big) >0,\xi_{t,x}\big(\eta_t|_{\textbf{B}(x,w_{2j-1,t,x})^c} \big)(B)=m \big) f(x) d x .
	\end{align*}
	Similar arguments as used to compute $M_t([u,\infty))$ for $u\in\R$ imply for $x\in W$ that 
	$$
	t \mathbb{P}\left(\eta_t\big(\textbf{B}(x, w_{2j-1,t,x})\big)=0\right) = e^{-u_{2j-1}},
	$$
	and so we deduce that 
	\begin{equation}\label{prof-step-2}\begin{split}
	S_j & =   e^{-u_{2j-1}}\int_{W}\mathbb{P}\left(\xi_{t,x}\big(\eta_t|_{\textbf{B}(x,w_{2j-1,t,x})^c}	\big)(B)=m \right) f(x) d x
	\\
	& \quad - e^{-u_{2j-1}}\int_{W}\mathbb{P}\big( \eta_t\big(\textbf{B}(x, w_{2j,t,x})\setminus \textbf{B}(x, w_{2j-1,t,x})\big) = 0,
	\\
	& \hspace{3.5cm} \xi_{t,x}\big(\eta_t|_{\textbf{B}(x,w_{2j-1,t,x})^c}\big)(B)=m \big) f(x) d x .
	\end{split}\end{equation}
	Furthermore, we can rewrite the second integral as 
	\begin{align*}
	& \int_{W}\mathbb{P}\left(\eta_t\big(\textbf{B}(x, w_{2j,t,x})\setminus \textbf{B}(x, w_{2j-1,t,x})\big) = 0 \right) \mathbb{P}\left(\xi_{t,x}\big(\eta_t|_{\textbf{B}(x,w_{2j,t,x})^c} \big)(B)=m\right) f(x) dx 
	\\
	& =      \int_{W} e^{-t \mu(\textbf{B}(x, w_{2j,t,x})) + t\mu (\textbf{B}(x, w_{2j-1,t,x})) }\mathbb{P}\left(\xi_{t,x}\big(\eta_t|_{\textbf{B}(x,w_{2j,t,x})^c}  \big)(B)=m\right) f(x) d x  
	\\
	& = e^{-u_{2j} + u_{2j-1}}  \int_{W} \mathbb{P}\left(\xi_{t,x}\big(\eta_t|_{\textbf{B}(x,w_{2j,t,x})^c}\big)(B)=m\right) f(x) d x .
	\end{align*}
	Combining this and \eqref{prof-step-2} yields
	\begin{align*}
	S_j & =  e^{-u_{2j-1}}\int_{W}\mathbb{P}\left(\xi_{t,x}\big(\eta_t|_{\textbf{B}(x,w_{2j-1,t,x})^c}	 \big)(B)=m \right) f(x) d x 
	\\
	& \quad - e^{-u_{2j} }  \int_{W} \mathbb{P}\left(\xi_{t,x}\big(\eta_t|_{\textbf{B}(x,w_{2j,t,x})^c}\big)(B)=m\right) f(x) d x.
	\end{align*}
Substituting this into \eqref{sum-step0} implies that to prove \eqref{prfthciin} and to complete the proof, it is enough to show  for all $x\in W$ and $k=1,\dots, 2\ell$ that
	\begin{align}\label{lem:last-step}
	\underset{t\rightarrow\infty}{\lim} \,\mathbb{P}\left(\xi_{t,x}\big(\eta_t|_{\textbf{B}(x,w_{k,t,x})^c}	\big)(B)=m \right)  - \mathbb{P}(\xi_t(B)=m ) =0.
	\end{align}
	Let $x\in W$, $k\in\{1,\dots,2\ell\}$ and $\varepsilon>0$ be fixed. Set
	$$
	a_t = 2\Big(\frac{ u_{2 \ell} + \log (t)}{ 2^d f_{min} k_d t }\Big)^{1/d}.
	$$
	By Lemma \ref{lm : exist and bound of v} there exists a $t'>0$ such that  $w_{k,t,y} \leq  w_{2\ell,t,y} \leq  a_t$ for all $y\in W$ and $t>t'$. Without loss of generality we may assume that $2a_t< \min\{ \Vert z_1-z_2 \Vert \, : \, z_1\in \partial W, z_2 \in\partial A\}$. Therefore the observation 
	$$
	h(y, \nu)\in B \quad \text{if and only if} \quad h(y, \nu|_{\mathbf{B}(y,w_{2\ell,t,y})})\in B
	$$
	for any point configuration $\nu$ on $\R^d$ with $y\in\nu$ leads to
	\begin{align*}
	& \big\vert    \mathbb{P}\left(\xi_{t,x}\big(\eta_t|_{\textbf{B}(x,w_{k,t,x})^c}	 \big)(B)=m \right) - \mathbb{P}\left(\xi_t(B)=m \right)   \big\vert
    \leq \mathbb{E}\big[\vert \xi_{t,x}\big(\eta_t|_{\textbf{B}(x,w_{k,t,x})^c} \big)(B) - \xi_t(B) \vert \big]
	\\
	& \leq  \mathbb{E}\sum_{y\in\eta_t \cap \mathbf{B}(x,2a_t)\cap \textbf{B}(x,w_{k,t,x})^c} \mathbf{1}\{h_t(y,\eta_t|_{\textbf{B}(x,w_{k,t,x})^c} + \delta_x)\in B\} + \mathbb{E}\sum_{y\in\eta_t \cap \mathbf{B}(x,2a_t)} \mathbf{1}\{h_t(y,\eta_t )\in B\}
	\\
	& \leq  \mathbb{E}\sum_{y\in\eta_t \cap \mathbf{B}(x,2a_t)\cap \textbf{B}(x,w_{k,t,x})^c} \mathbf{1}\{h_t(y,\eta_t|_{\textbf{B}(x,w_{k,t,x})^c} + \delta_x) >u_1\} 
	\\
	& \quad + \mathbb{E}\sum_{y\in\eta_t \cap \mathbf{B}(x,2a_t)} \mathbf{1}\{h_t(y,\eta_t ) >u_1\}.
	\end{align*}
	Then, the Mecke formula and  \eqref{h_t-sec-1.3.3}  imply that
	\begin{align*}
	& \big\vert    \mathbb{P}\left(\xi_{t,x}\big(\eta_t|_{\textbf{B}(x,w_{k,t,x})^c} \big)(B)=m \right) - \mathbb{P}\left(\xi_t(B)=m \right)   \big\vert
	\\
	& \leq t \int_{\mathbf{B}(x,2a_t)\cap \textbf{B}(x,w_{k,t,x})^c} \mathbb{P}(h_t(y, \eta_t|_{\textbf{B}(x,w_{k,t,x})^c} +\delta_x + \delta_y)>u_1) f(y) dy
	\\
	& \quad + t \int_{\mathbf{B}(x,2a_t)} \mathbb{P}(h_t(y, \eta_t + \delta_y) >u_1) f(y) dy
	\\
	& = t \int_{\mathbf{B}(x,2a_t)\cap \textbf{B}(x,w_{k,t,x})^c} \mathbb{P}\big(t \mu(\mathbf{B}(y,2c(y, \eta_t|_{\textbf{B}(x,w_{k,t,x})^c} +\delta_x + \delta_y)))  > u_1 + \log(t)\big) f(y)dy 
	\\
	& \quad + t \int_{\mathbf{B}(x,2a_t)} \mathbb{P}\big(t \mu(\mathbf{B}(y,2c(y,\eta_t  + \delta_y)))  > u_1 + \log(t)\big) f(y)dy .
	\end{align*}
	Since $c(y,\nu + \delta_y + \delta_x)> v_t(y,u_1)$ only if $c(y,\nu + \delta_y )> v_t(y,u_1)$ for any point configuration $\nu$ on $\R^d$ and $x,y\in \R^d$, it follows for $x\in W$ and $y\in \mathbf{B}(x,2a_t)\cap \textbf{B}(x,w_{k,t,x})^c$  that
	\begin{align*}
	& \mathbb{P}\big(t \mu(\mathbf{B}(y,2c(y, \eta_t|_{\textbf{B}(x,w_{k,t,x})^c} +\delta_x + \delta_y)))  > u_1 + \log(t)\big) 
	\\
	& =  \mathbb{P}\big(c(y,\eta_t|_{\textbf{B}(x,w_{k,t,x})^c} +\delta_x + \delta_y)> v_t(y,u_1)\big)\leq  \mathbb{P}\big(c(y,\eta_t|_{\textbf{B}(x,w_{k,t,x})^c} + \delta_y)> v_t(y,u_1)\big)
	\\
	& \leq \mathbb{P}\big( \eta_t|_{\textbf{B}(x,w_{k,t,x})^c}(\textbf{B}(y, 2v_t(y,u_1)))=0\big) = \exp(-t\mu(\textbf{B}(y, 2v_t(y,u_1))\cap \textbf{B}(x,w_{k,t,x})^c)).  
	\end{align*}
	Let $\lambda_d$ denote the Lebesgue measure on $\R^d$. For $A_1,A_2\in\mathcal{B}(\R^d)$ with $A_1,A_2\subset A$ and $\lambda_d(A_2)>0$ we obtain 
	$$
	\frac{\mu(A_1)}{\mu(A_2)} \geq \frac{f_{min}}{f_{max}} \frac{\lambda_d(A_1)}{\lambda_d(A_2)}.
	$$
	With $\tau:=f_{min}/f_{max}\in(0,1]$, 
	$$
	A_1 =\textbf{B}(y, 2v_t(y,u_1))\cap \textbf{B}(x,w_{k,t,x})^c \quad and \quad A_2 =\textbf{B}(y, 2v_t(y,u_1)),
	$$
	  this implies for $x\in W$ and $y\in \mathbf{B}(x,2a_t)\cap \textbf{B}(x,w_{k,t,x})^c$ that
	$$
	t\mu(\textbf{B}(y, 2v_t(y,u_1))\cap \textbf{B}(x,w_{k,t,x})^c) \geq \frac{\tau}{2} t \mu(\textbf{B}(y, 2v_t(y,u_1))) = \frac{\tau }{2} (u_1+\log(t)).
	$$
	Moreover, we have that
	\begin{align*}
	& \mathbb{P}\big(t \mu(\mathbf{B}(y,2c(y, \eta_t + \delta_y)))  > u_1 + \log(t)\big) 
	= \mathbb{P}\big( \eta_t(\textbf{B}(y, 2v_t(y,u_1)))=0\big) = e^{-u_1-\log(t)}.
	\end{align*}
	In conclusion, combining the previous bounds leads to
     \begin{align*}
   	& \big\vert    \mathbb{P}\left(\xi_{t,x}\big(\eta_t|_{\textbf{B}(x,w_{k,t,x})^c} \big)(B)=m \right) - \mathbb{P}\left(\xi_t(B)=m \right)   \big\vert \\
   	&\leq t^{1-\tau/2} e^{-\tau u_1/2}\mu(\mathbf{B}(x,2a_t)) + e^{-u_1}\mu(\mathbf{B}(x,2a_t))
    \leq (2a_t)^d  k_d f_{max} [t^{1-\tau/2} e^{- \tau u_1 /2} + e^{-u_1}],
    \end{align*}
    where in the last step we used  the fact that $f$ is bounded by $f_{max}$ in $A$ and, by  the choice of $a_t$, $\mathbf{B}(x,2a_t) \subset A$. Again, from the definition of $a_t$ it follows that the right-hand side converges to $0$ as $t\to\infty$. This shows \eqref{lem:last-step} and concludes the proof.	
\end{proof}
Next, we derive Theorem \ref{Thm:inr-cont-dens} from Theorem \ref{thm: max inradius}. 
\begin{proof}[Proof of Theorem \ref{Thm:inr-cont-dens}]
Assume that $f$ is H\"older  continuous with   exponent $b>0$. From Lemma \ref{lem: equiv-Wconv}, Theorem \ref{thm: max inradius} and Remark \ref{rem:Euclidean} we obtain that it is enough to show that $\mathbb{E}[\vert \xi_t(B)-  \widehat{\xi}_t(B)\vert ] \rightarrow 0$ as $t\to\infty$ for all  $B\in\mathcal{I}$. 
By Lemma \ref{lm : exist and bound of v}, for any $u\in\R$ there exists a $t_0>0$ such that 
\begin{equation*}\begin{split}
&  \mu(\mathbf{B}(x, 2v_t(x,u))) = 2^d  \mu(\mathbf{B}(x, q_t(x,u)))
\\
& = 2^d  k_d f(x) q_t(x,u)^d + 2^d  \int_{\mathbf{B}(x,q_t(x,u))}f(y)-f(x) dy 
\\
& = \mu(\mathbf{B}(x, 2q_t(x,u))) - \int_{\mathbf{B}(x,2q_t(x,u))}f(y)-f(x) dy + 2^d  \int_{\mathbf{B}(x,q_t(x,u))}f(y)-f(x) dy
 \end{split}\end{equation*}
for all $x\in W, t>t_0$, where
\begin{align*}
	\max\{v_t(x,u), q_t(x,u)\}\leq \Big(\frac{ u + \log (t)}{ 2^d f_{min} k_d t }\Big)^{1/d}.
\end{align*}
Thus,  the H\"older  continuity of $f$  and elementary arguments establish that 
\begin{align}\label{eqn: diff-mu v- mu q}
\vert  \mu(\mathbf{B}(x, 2v_t(x,u)))-   \mu(\mathbf{B}(x, 2q_t(x,u)))\vert \leq C\Big(\frac{ u + \log(t)}{  t}\Big)^{1+b/d}, \quad x\in W, t>t_0,
\end{align} 
for some $C>0$. In particular, from the definition of $v_t(x,u)$ it follows  that  
\begin{equation}\label{eqn: bnd mu q}\begin{split}
 &\mu(\mathbf{B}(x, 2v_t(x,u)))=\frac{u+\log(t)}{t} 
\\
& \mu(\mathbf{B}(x, 2q_t(x,u)))\geq \frac{u+\log(t)}{t}	 - C\Big(\frac{ u + \log(t)}{  t}\Big)^{1+b/d}
\end{split}\end{equation}
for $t>t_0$.  
  Next, we write $B=\bigcup_{j=1}^{\ell} (u_{2j-1}, u_{2j})$ for some $\ell\in\N$ and $u_1<\cdots <u_{2\ell}$.
The triangle inequality yields
\begin{equation}\label{eq:liminL1xi-hatxi}\begin{split}
& \mathbb{E}[\vert \xi_t(B)-  \widehat{\xi}_t(B)\vert ]  \leq \sum_{j=1}^{\ell} \mathbb{E}[\vert \xi_t((u_{2j-1},u_{2j}))-  \widehat{\xi}_t((u_{2j-1},u_{2j}))\vert ] 
\\
& \leq \sum_{j=1}^{\ell} \mathbb{E}[\vert \xi_t((u_{2j-1},\infty))-  \widehat{\xi}_t((u_{2j-1},\infty))\vert ] + \mathbb{E}[\vert \xi_t([u_{2j},\infty))-  \widehat{\xi}_t([u_{2j},\infty))\vert].
\end{split}\end{equation}
Moreover, the Mecke formula establishes for   $u\in\R$ that  
\begin{align*}
& \mathbb{E}[\vert \xi_t((u,\infty))-  \widehat{\xi}_t((u,\infty))\vert ]  
\\
& \leq  \mathbb{E}\sum_{x\in\eta_t \cap W}\vert\mathbf{1}\{c(x,\eta_t + \delta_x )> v_t(x,u)\}
 -\mathbf{1}\{c(x,\eta_t + \delta_x )> q_t(x,u)\}\vert 
\\
& = t\int_W \mathbb{P}\big(\eta_t(\mathbf{B}(x,2v_t(x,u)))=0, \eta_t(\mathbf{B}(x,2q_t(x,u)))>0 \big) f(x) d x
\\
&\quad + t\int_W \mathbb{P}\big(\eta_t(\mathbf{B}(x,2v_t(x,u)))>0, \eta_t(\mathbf{B}(x,2q_t(x,u)))=0 \big) f(x) d x
\\
& \leq f_{max} t\int_{W}\big[\exp\big(-t\mu(\mathbf{B}(x,2v_t(x,u)))\big)+ \exp\big(-t\mu(\mathbf{B}(x,2q_t(x,u)))\big)\big]
\\
& \hspace{2cm}\times\big[1 - \exp\big(-t\vert\mu(\mathbf{B}(x,2q_t(x,u)))- \mu(\mathbf{B}(x,2v_t(x,u)))\vert \big)\big]  dx.
\end{align*}
Therefore, from \eqref{eqn: diff-mu v- mu q} and \eqref{eqn: bnd mu q}, it follows that
\begin{align}\label{eqn:limL1-inr}
\lim_{t\rightarrow\infty} \mathbb{E}[\vert \xi_t((u,\infty))-  \widehat{\xi}_t((u,\infty))\vert ] = 0.
\end{align}
Together with \eqref{eq:liminL1xi-hatxi} and a similar computation for the half-closed intervals on the right-hand side of \eqref{eq:liminL1xi-hatxi}, this concludes the proof.
\end{proof}
\begin{proof}[Proof of Corollary \ref{cor:maxInr}]
Let $u\in\R$ be fixed. By Markov's inequality we have for $u_0>u$ that
	\begin{align*}
		\mathbb{P}(\xi_t((u,u_0))>0 ) &\leq \mathbb{P}(\xi_t((u,\infty))>0 ) = \mathbb{P}\big( \max_{x\in\eta_t\cap W} t \mu(\mathbf{B}(x,2c(x,\eta_t ))) - \log(t) > u  \big) 
		\\
		& \leq  \mathbb{P}(\xi_t((u,u_0))>0 ) + \mathbb{E}[\xi_t([u_0,\infty))].
	\end{align*}
	Thus, Theorem \ref{thm: max inradius} and  \eqref{int-mea-xi}   yield
	$$
	 \limsup_{t\rightarrow\infty} \vert  \mathbb{P}\big( \max_{x\in\eta_t\cap W} t \mu(\mathbf{B}(x,2c(x,\eta_t ))) - \log(t) > u  \big) - 1 + e^{-M((u,u_0))} \vert \leq e^{-u_0}.
	$$
	Then, letting $u_0\to \infty$ leads to \eqref{eq:max-inr}. Since, for $u>0$,
	$$
	\big|\mathbb{P}(\xi_t((u,\infty))>0 ) - \mathbb{P}(\widehat{\xi}_t((u,\infty))>0 ) \big| \leq \mathbb{E}[\vert \xi_t((u,\infty))-  \widehat{\xi}_t((u,\infty))\vert],
	$$
	\eqref{eq:max-inr} and \eqref{eqn:limL1-inr} imply \eqref{eq:max-inr-contf}.
\end{proof}

\subsection{Circumscribed radii of an inhomogeneous Poisson Voronoi tessellation}\label{circ-rad-proof}
In this last section, we consider  the  circumscribed radii of an inhomogeneous  Voronoi tessellation generated by a Poisson  process with a certain intensity measure $t\mu, t>0$;  recall that the circumscribed radius of a cell is  the smallest radius for which the ball centered at the nucleus contains the cell. We study the point process on the non-negative real line constructed by taking  for any cell with the nucleus in a compact convex set, a transform of the $\mu$-measure of the ball centered at the nucleus and with  the circumscribed radius as the radius. The aim is to continue the work started in \cite{MR3252817} by extending the  result on the smallest circumscribed radius to  inhomogeneous  Poisson-Voronoi tessellations and by   proving weak convergence of    the aforementioned  point process to a Poisson process. 

More precisely, let $\mu$ be an absolutely continuous  measure on $\R^d$  with continuous density $f: \R^d \rightarrow [0,\infty)$.   Consider a Poisson  process $\eta_{t}$  with intensity measure $t \mu$, $t>0$.
The circumscribed radius of the Voronoi cell $N(x,\eta_t)$ with $x\in\eta_t$ is given by
$$
C(x,\eta_t)
=\inf\left\{R\geq0\,:\,\mathbf{B}(x,R)\supset N(x,\eta_t)\right\},
$$
with the convention $\inf\emptyset =\infty$; see Section \ref{inradius} for more details on Voronoi tessellations.

Let $W\subset\mathbb{R}^d$ be a compact convex set with $f>0$ on $W$. We consider the point process
\begin{align}\label{xi-circ-rad}
\xi_t
=\sum_{x\in\eta_t \cap W}\delta_{\alpha_2t^{(d+2)/(d+1)} \mu(\mathbf{B}(x,C(x,\eta_t)))}.
\end{align}
Here the positive constant $\alpha_2$ is given by
$$
\alpha_2
=\left(\frac{2^{d(d+1)}}{(d+1)!}p_{d+1}\right)^{1/(d+1)}
$$
with
$$
p_{d+1} := \mathbb{P}\Big( N\Big(0,\sum_{j=1}^{d+1} \delta_{Y_j}\Big) \subseteq  \mathbf{B}(0,  1)\Big),
$$
where $Y_1, \dots , Y_{d+1}$ are independent and uniformly distributed random points in $\mathbf{B}(0, 2)$. We write $M$ for the measure on $[0,\infty)$ satisfying $M([0,u])= \mu(W) u^{d+1}$ for $u\geq0$.
\begin{theorem}\label{main-Thm-circ}
	Let $\xi_t$, $t>0$, be the family of point processes on $[0,\infty)$ given by \eqref{xi-circ-rad}. Then $\xi_t$ converges in distribution to a Poisson  process  with intensity measure $M$.
\end{theorem}
An immediate consequence of this theorem is that a transform of the minimal $\mu$-measure of the balls, having  circumscribed radii and  nuclei of the Voronoi cells  as radii and centers respectively, converges to a  Weibull distributed random variable. This generalizes \cite[Theorem 1, Equation (2d)]{MR3252817}. For the situation that, in contrast to Theorem \ref{main-Thm-circ}, the density of the intensity measure of the underlying Poisson process is not continuous, we can still derive some upper and lower bounds.
\begin{theorem}\label{min circradii}
Let $\zeta_t$ be a Poisson  process with intensity measure $t\vartheta$, where $t>0$ and $\vartheta$ is an absolutely continuous measure on $\R^d$ with density $\phi$.  Let $ f_1, f_2:\R^d \rightarrow [0,\infty)$ be  continuous and $f_1,f_2>0$ on $W$. 
\begin{itemize}
\item[(i)]
If there exists $s\in(0,1]$ such that $s\phi\leq f_1 \leq \phi$, then
$$
\limsup_{t\rightarrow\infty}\mathbb{P}\Big( s\alpha_2t^{(d+2)/(d+1)}\min_{x\in\zeta_t\cap W} \vartheta(\mathbf{B}(x, C(x,\zeta_t))) >u \Big)\leq  \exp\big(- s \vartheta(W) u^{d+1}\big)
$$
for $u\geq0$.
\item[(ii)]
If there exists $r\geq 1$ such that $\phi \leq f_2 \leq  r\phi$, then
$$
\liminf_{t\rightarrow\infty}\mathbb{P}\Big( r\alpha_2t^{(d+2)/(d+1)}\min_{x\in\zeta_t\cap W} \vartheta(\mathbf{B}(x, C(x,\zeta_t))) >u \Big)\geq  \exp\big(-r \vartheta(W) u^{d+1}\big)
$$
for $u\geq0$.
\end{itemize}
\end{theorem}

Let us now prepare the proof of Theorem \ref{main-Thm-circ}.  We first have to study the distribution of $C(x,\eta_t + \delta_x)$, which is defined as the circumscribed radius of the Voronoi cell with nucleus $x\in\R^d$ generated by $\eta_t + \delta_x$. To this end, we define  $g: W \times T \rightarrow [0,\infty)$ by the equation 
\begin{align}\label{eq: ident mu}
  \mu(\mathbf{B}(x,g(x,u)))=u
\end{align}
for $T:= [0, \mu(W)]$. Since $W$ is compact and convex and $f>0$ on $W$, we have that \eqref{eq: ident mu} admits a unique solution $g(x,u)$ for all $(x,u)\in W\times T$. As this is the only place where we use the convexity of $W$, we believe that one can omit this assumption. However, we refrained from doing so in order to not further increase the complexity of the proof. Set
\begin{align*}
s_t = \alpha_2 t^{(d+2)/(d+1)} .
\end{align*}
Thus, we may write 
\begin{align}\label{eq: def g}
\mathbb{P}(s_t\mu(\mathbf{B}(x,C(x,\eta_t+\delta_x)))\leq u)
= \mathbb{P}(C(x,\eta_t+ \delta_x)\leq g(x,u/s_t)) , \quad u/s_t \in T.
\end{align}

\begin{lemma}\label{lem: unifcont of g}
	For any $u\in T$, $g(\cdot, u ): W \rightarrow\R$ is continuous and 
	$$
	\lim_{u\rightarrow 0^+} \sup_{x\in W} \vert g(x,u)\vert =0.
	$$
\end{lemma}
\begin{proof}
	First we show that  $g(\cdot,u)$ is continuous for any fixed $u\in T$. For $u=0$, we obtain $g(x,u)=0$ for all $x\in W$. Assume $u>0$ and let $x_0\in W$ and $\varepsilon >0$. Then for all  $x\in \mathbf{B}(x_0, \varepsilon')$ with $\varepsilon':= \min\{g(x_0, u)/2, \varepsilon\}$, we have that 
    $$
    \mathbf{B}(x_0, g(x_0, u)) \subset \mathbf{B}(x, g(x_0, u)+  \varepsilon') \quad \text{ and } \quad \mathbf{B}(x, g(x_0, u)- \varepsilon')\subset\mathbf{B}(x_0, g(x_0, u)).
    $$
    Together with  \eqref{eq: ident mu}, this leads to 
    $$
    \mu(\mathbf{B}(x, g(x_0, u)+  \varepsilon')) \geq u \quad \text{ and } \quad \mu(\mathbf{B}(x, g(x_0, u)- \varepsilon')) \leq u.
    $$
    Now it follows from the definition of $g$ that
    $$
    g(x,u) \leq g(x_0,u)+\varepsilon' \quad \text{ and } \quad g(x,u) \geq g(x_0,u)-\varepsilon'.
    $$
   This yields
    $$
    \vert g(x,u)- g(x_0,u) \vert \leq \varepsilon' \leq \varepsilon
    $$
    for all $x\in \mathbf{B}(x_0, \varepsilon')$ so that $g$ is continuous at $x_0$.  In conclusion since
    $$
    \lim_{u\rightarrow 0^+} g(x,u)=0
    $$
     and  $g(x,u_1)< g(x,u_2)$ for all  $x\in W$ and $0\leq u_1 < u_2$, Dini's theorem implies that $\sup_{x\in W} \vert g(x,u)\vert \to 0$ as $u\to 0$.
\end{proof}
We define
$$
\beta = \min_{x\in W} f(x)>0.
$$
\begin{lemma}\label{lem:bnd on g} 
There exists $u_0\in T$ such that
\begin{align}\label{eq: bound on g}
g(x,u)\leq \Big(\frac{2u}{\beta  k_d}\Big)^{1/d}
\end{align}
for all $x\in W$. 
\end{lemma}
\begin{proof}
Since $f$ is continuous and  $f>0$ on $W$, it follows that 
$$
\min_{x\in W + \overline{\mathbf{B}(0, \delta )}} f(x) > \frac{\beta}{2}
$$
for some $\delta >0$. Furthermore, by Lemma \ref{lem: unifcont of g}  we obtain that there exists $u_0\in T$ such that  $g(x, u) \leq \delta$ for all $u\in [0,u_0]$ and $x\in W$. 
Then, we obtain
$$
u =\mu(\mathbf{B}(x,g(x,u)))= \int_{\mathbf{B}(x,g(x,u))} f(y) dy \geq \frac{\beta}{2} k_d g(x,u)^d
$$
for all  $x\in W$ and $u\in[0,u_0]$, which shows \eqref{eq: bound on g}.
\end{proof}
For $x\in W$ and $u\geq 0$, we consider a sequence of independent and identically distributed random points $(X^{(x,u)}_i)_{i\in\N}$ in $\R^d$ with distribution 
$$
\mathbb{P}\big(X^{(x,u)}_i\in E\big)= \frac{\mu(\mathbf{B}(x,2u)\cap E) }{\mu(\mathbf{B}(x,2u))}, \quad i\in\N,  E\in \mathcal{B}(\R^d).
$$
Recall that, for $k\geq d+1$, $N\Big(x,\sum_{j=1}^k \delta_{X^{(x,u)}_j}\Big)$ denotes the Voronoi cell with nucleus $x$ generated by $X^{(x,u)}_1,\dots , X^{(x,u)}_k$ and $x$. 
Then the distribution function of $C(x,\eta_t+\delta_x)$ is equal to  
\begin{equation}\label{ff11}
\mathbb{P}\big(C(x,\eta_t+\delta_x)\leq u\big)=\sum_{k=d+1}^\infty \mathbb{P}\big(\eta_t (\mathbf{B}(x,2u))=k\big) p_k(x,u)
\end{equation}
for $u\geq0$,  with $p_k(x,u)$ defined as  
$$
p_k(x,u) = \mathbb{P}\Big( N\Big(x,\sum_{j=1}^k \delta_{X^{(x,u)}_j}\Big) \subseteq  \mathbf{B}(x,  u)\Big).
$$
Combining \eqref{eq: def g} and \eqref{ff11} establishes for $u/s_t\in T$ that 
\begin{equation}\begin{split}\label{eq: distr smu(C)}
&\mathbb{P}(s_t\mu(\mathbf{B}(x,C(x,\eta_t+\delta_x)))\leq u)
\\
&= \sum_{k=d+1}^\infty \mathbb{P}\big(\eta_t (\mathbf{B}(x,2g(x,u/s_t)))=k\big) p_k(x,g(x,u/s_t)).
\end{split}\end{equation}
For $k\in \N$ with $k\geq d+1$, we define the probability
$$
p_k = \mathbb{P}\Big( N\Big(0,\sum_{j=1}^k \delta_{Y_j}\Big) \subseteq  \mathbf{B}(0,  1)\Big),
$$
where $Y_1, \dots , Y_k$ are independent and uniformly distributed random points in $\mathbf{B}(0, 2)$. As discussed in \cite[Section 3]{MR3252817}, one can reinterpret $p_k$ as  the probability to cover the unit sphere with $k$ independent spherical caps with random radii. In the next lemma, we prove that $p_k(x,r)\to p_k$ as $r\rightarrow0$ for all $x\in W$  and $k\geq d+1$, which together with Lemma \ref{lem:bnd on g} yields $p_k(x, g(u/s_t))\to p_k$ as $t\rightarrow \infty$. 
\begin{lemma}\label{lem: lim p(x,k)}
For any $k\geq d+1$ and $x\in W$, 
$$
\lim_{r\rightarrow 0^+} p_k(x,r) = p_k.
$$
\end{lemma}
\begin{proof}
 In this proof, to simplify the notation,  for any $x\in W$, $k\geq d+1$ and $y_1,\hdots,y_k\in\mathbb{R}^d$, we denote by $K_k^{(x)} (y_1,\dots ,y_k)$ the Voronoi cell $N\big(x,\sum_{j=1}^k \delta_{y_j}\big)$  with nucleus   $x$  generated by $y_1,\dots , y_k$ and $x$. Thus,  we may write
$$
p_k(x,r)
= \mathbb{P}\big( K_k^{(x)} \big(X^{(x,r)}_1,\dots , X^{(x,r)}_k\big)\subseteq  \mathbf{B}(x,  r)\big),
$$
and so from the independence of $X^{(x,r)}_1,\dots , X^{(x,r)}_k$  it follows that
\begin{align*}
 p_k(x,r) 
&= \frac{1}{\mu(\mathbf{B}(x, 2r))^k} \int_{\mathbf{B}(x, 2r)^k} \mathbf{1}\big\{K_k^{(x)} (z_1,\dots , z_k) \subseteq \mathbf{B}(x, r) \big\}
\prod_{i=1}^k f(z_i) \, dz_1\, \dots dz_k
\\
&= \frac{(2r)^{k d}}{\mu(\mathbf{B}(x, 2r))^k}\int_{\mathbf{B}(0, 1)^k} \mathbf{1}\big\{K_k^{(x)} (x+2rz_1,\dots ,x+ 2rz_k) \subseteq \mathbf{B}(x, r) \big\}  \\
&\hspace{7cm} \times\prod_{i=1}^k f(x+ 2rz_i) \, dz_1\, \dots dz_k.
\end{align*}
Furthermore, by the definition of $K_k^{(x)}$ we deduce that
$$
\mathbf{1}\big\{K_k^{(x)} (x+2rz_1,\dots ,x+ 2rz_k) \subseteq \mathbf{B}(x, r) \big\} =
\mathbf{1}\big\{K_k^{(0)} (2z_1,\dots , 2 z_k) \subseteq \mathbf{B}(0, 1) \big\}
$$ 
for all $z_1,\dots , z_k \in \mathbf{B}(0,1)$, whence
\begin{align*}
p_k(x,r) 
&
= \frac{(2r)^{k d}}{\mu(\mathbf{B}(x, 2r))^k}\int_{\mathbf{B}(0, 1)^k} \mathbf{1}\big\{K_k^{(0)} (2z_1,\dots , 2 z_k) \subseteq \mathbf{B}(0, 1) \big\} 
\\
& \hspace{6cm}\times \prod_{i=1}^k f(x+ 2rz_i) \, dz_1\, \dots dz_k.
\end{align*}
Using the dominated convergence theorem for the integral, the continuity of $f$ and
$$
\lim_{t\rightarrow\infty} \frac{(2r)^{k d}}{\mu(\mathbf{B}(x, 2r))^k} = \frac{1}{k_d^{k} f(x)^k},
$$
we obtain
\begin{align*}
\lim_{t\rightarrow\infty}p_k(x,r) 
&
= \frac{1}{k_d^k}\int_{\mathbf{B}(0, 1)^k} \mathbf{1}\big\{K_k^{(0)} (2z_1,\dots , 2 z_k) \subseteq \mathbf{B}(0, 1) \big\}  \, dz_1\, \dots dz_k \\
& = \frac{1}{(2^d k_d)^k}\int_{\mathbf{B}(0, 2)^k} \mathbf{1}\big\{K_k^{(0)} (z_1,\dots ,  z_k) \subseteq \mathbf{B}(0, 1) \big\}  \, dz_1\, \dots dz_k = p_k ,
\end{align*}
which concludes the proof.
\end{proof}

Let $M_t$ be the intensity measure of $\xi_t$ and let
\begin{align*}
 \widehat{M}_t([0,u]) 
& = t \int_{W}\,\mathbb{E}\Big[\mathbf{1}\big\{s_t \mu(\mathbf{B}(x , C(x,\eta_{t} +\delta_x)))\in [0,u]\big\} \\
& \hspace{2cm}\times\mathbf{1}\Big\{\eta_t\Big(\mathbf{B}\Big(x, 4\Big(\frac{2u}{\beta s_t k_d}\Big)^{1/d}\Big)\Big)=d+1\Big\} \Big] f(x)dx
\end{align*}
and
\begin{align*}
\theta_t ([0,u]) 
& =  t \int_{W}\,\mathbb{E}\Big[\mathbf{1}\big\{ s_t \mu(\mathbf{B}(x , C(x,\eta_{t} +\delta_x)))\in [0,u]\big\}
\\
& \hspace{2cm}\times \mathbf{1}\Big\{\eta_t\Big(\mathbf{B}\Big(x, 4\Big(\frac{2 u}{ \beta s_t k_d}\Big)^{1/d}\Big)\Big)>d+1\Big\} \Big] f(x) dx
\end{align*}
for $u\geq 0$. Observe that 
\begin{align}\label{M()}
M_t([0,u])
= \widehat{M}_t([0,u]) + \theta_t ([0,u]), \quad u\geq0.
\end{align}
\begin{lemma}\label{calc-meas}
For any $u\geq 0$,
	\begin{align*}
\lim_{t\rightarrow\infty}\widehat{M}_t([0,u])  
 = \mu(W)u^{d+1}
	\end{align*}
	and
	$$
	\theta_t ([0,u]) \leq t \int_{W}\,  \mathbb{P}\Big(\eta_t\Big(\mathbf{B}\Big(x, 4\Big(\frac{2 u}{ \beta s_t k_d}\Big)^{1/d}\Big)\Big)>d+1\Big) f(x) dx \to 0 \quad \text{as} \quad t\to\infty.
	$$
	
\end{lemma}
\begin{proof}
	Let $u\geq0$ be fixed and $u_t:=u/s_t$. Without loss of generality we may assume $u_t\in T$. 
For $x\in W$ we deduce from \eqref{eq: ident mu}, $g(x,u_t)\to 0$ as $t\to\infty$ and the continuity of $f$ that
	$$
	\lim_{t\to\infty} \frac{\mu(\mathbf{B}(x,2g(x,u_t)))}{u_t} = \lim_{t\to\infty} \frac{\mu(\mathbf{B}(x,2g(x,u_t)))}{2^d k_d g(x,u_t)^d} \frac{2^d k_d g(x,u_t)^d}{\mu(\mathbf{B}(x,g(x,u_t)))} = \frac{2^d f(x)}{f(x)} = 2^d.
	$$
	Together with $u_t=u/s_t$ and $s_t=\alpha_2 t^{(d+2)/(d+1)}$, this leads to
	\begin{equation}\label{eq: lim t to k+1 times mu to k}
\lim_{t\rightarrow\infty}t^{d+2}\mu\big(\mathbf{B}(x,2g(x,u_t))\big)^{d+1}= ( 2^{d}u / \alpha_2 )^{d+1}.
	\end{equation}
Similarly, we obtain from Lemma \ref{lem:bnd on g} that for $t$ sufficiently large, 
	\begin{equation}
	\begin{split}\label{eq:bnd on mu}
	\sup_{x\in W} t^{d+2}\mu\big(\mathbf{B}(x,2g(x,u_t))\big)^{d+1} & \leq t^{d+2} (2^{d+1} u_t/\beta )^{d+1} \sup_{x\in W}\sup_{y\in \mathbf{B}(x, 2g(x,u_t))} f(y)^{d+1} \\
	& \leq ( 2^{d+1}u / (\alpha_2 \beta) )^{d+1} \sup_{y\in W+\mathbf{B}(0, 1)} f(y)^{d+1}.
	\end{split}
	\end{equation}
		Let us now compute the limit of $\widehat{M_t}([0,u])$. By Lemma \ref{lem:bnd on g} we obtain  for $\ell_t : = 4\big(\frac{2u_t}{\beta  k_d}\big)^{1/d}$ that there exists $t_0>0$ such that $2 g(x,u_t)\leq \ell_t$ for all $t>t_0$ and $x\in W$. From \eqref{eq: distr smu(C)} we deduce for $x	\in W$  that $s_t \mu(\mathbf{B}(x, C(x,\eta_{t} +\delta_x)))\in [0,u]$ only if there are at least $d+1$ points of $\eta_t$ in $\mathbf{B}\big(x,2g(x,u_t)\big)$. Then for $t>t_0$, we have
	\begin{equation*}\begin{split}
	\widehat{M}_t([0,u]) 
	& = t \int_{W}  \mathbb{P}\big(\eta_t (\mathbf{B}(x,2g(x,u_t)))=d+1\big) p_{d+1}(x,g(x,u_t))
	\\
	& \hspace{1.5cm} \times \mathbb{P}\big(\eta_t\big(\mathbf{B}(x,  \ell_t)\setminus\mathbf{B}(x,2g(x,u_t))\big)=0\big) f(x) dx
	\\
	& =\int_W \frac{t^{d+2}\mu\big(\mathbf{B}(x,2g(x,u_t))\big)^{d+1}}{(d+1)!}e^{-t\mu(\mathbf{B}(x,\ell_t))}p_{d+1}(x,g(x,u_t)) f(x) dx.
	\end{split}\end{equation*}
Elementary arguments imply  that
$$
\lim_{t\rightarrow\infty} t\mu(\mathbf{B}(x,\ell_t))=0 .
$$
Therefore combining \eqref{eq: lim t to k+1 times mu to k} and Lemma \ref{lem: lim p(x,k)}  yields
\begin{align*}
&\lim_{t\rightarrow\infty}\frac{t^{d+2}\mu\big(\mathbf{B}(x,2g(x,u_t))\big)^{d+1}}{(d+1)!}e^{-t\mu(\mathbf{B}(x,\ell_t))}p_{d+1}(x,g(x,u_t)) f(x)
\\
& 
= \Big(\frac{2^d u }{\alpha_2}\Big)^{d+1} \frac{p_{d+1}}{(d+1)!} f(x) = u^{d+1} f(x),
\end{align*}
where we used $\alpha_2= \big(\frac{2^{d(d+1)}}{(d+1)!}p_{d+1}\big)^{1/(d+1)}$ in the last step. Thus, by \eqref{eq:bnd on mu} and the dominated convergence theorem we obtain
$$
\lim_{t\rightarrow\infty}\widehat{M}_t([0,u])  
= u^{d+1}\int_{W} f(x) dx
=   \mu(W) u^{d+1}.
$$ 
Finally, let us compute the limit of $\theta_t ([0,u])$. For a Poisson distributed random variable $Z$ with parameter $v>0$ we have
$$
\mathbb{P}(Z\geq d+2) = \sum_{k=d+2}^\infty \frac{v^k}{k!} e^{-v} \leq v^{d+2} \sum_{k=0}^\infty \frac{v^k}{k!} e^{-v} = v^{d+2}.
$$
This implies that 
\begin{align*}
\theta_t ([0,u]) & \leq t \int_W \mathbb{P}\bigg( \eta_t\Big(\mathbf{B}\Big(x, 4\Big(\frac{2u_t}{\beta k_d}\Big)^{1/d}\Big)\Big)>d+1 \bigg) f(x) dx \\
& \leq t^{d+3} \int_W \mu\Big(\mathbf{B}\Big(x, 4\Big(\frac{2u_t}{\beta k_d}\Big)^{1/d}\Big)\Big)^{d+2} f(x) dx \\
& \leq \sup_{y \in W+\mathbf{B}\Big(0, 4\Big(\frac{2u_t}{\beta k_d}\Big)^{1/d}\Big)} f(y) \int_W f(x) dx \frac{2^{2d^2+5d+2}}{\beta^{d+2}} t^{d+3} u_t^{d+2} \\
& = \sup_{y \in W+\mathbf{B}\Big(0, 4\Big(\frac{2u_t}{\beta k_d}\Big)^{1/d}\Big)} f(y) \mu(W) \frac{2^{2d^2+5d+2}}{\beta^{d+2}} \frac{1}{\alpha_2^{d+2}} t^{-\frac{1}{d+1}} u^{d+2}.
\end{align*}
Here, the supremum converges to a constant as $t\to \infty$ so that the second inequality in the assertion is proven.
\end{proof}

In the next lemma, we show a technical result, which will be needed in the proof of Theorem \ref{main-Thm-circ}. For $A\subset \mathbb{R}^d$, let $\operatorname{conv}(A)$ denote the  convex hull of $A$.
\begin{lemma}\label{lemmconvhull}
	Let $x_0,\hdots,x_{d+1}\in\mathbb{R}^d$ be in general position (i.e.\ no $k$-dimensional affine subspace of $\mathbb{R}^d$ with $k\in\{0,\hdots,d-1\}$ contains more than $k+1$ of the points) and assume that $N(x_0, \sum_{j=0}^{d+1} \delta_{x_j})$ is bounded. Then,
	\begin{itemize}
		\item [a)] $x_0\in\operatorname{int}(\operatorname{conv}(\{x_1,\hdots,x_{d+1}\}))$;
		\item [b)] $N(x_i, \sum_{j=0}^{d+1} \delta_{x_j})$ is unbounded for any $i\in\{1,\hdots,d+1\}$.
	\end{itemize}
\end{lemma}
\begin{proof}
	Assume that $x_0\notin\operatorname{int}(\operatorname{conv}(\{x_1,\hdots,x_{d+1}\}))$. By the hyperplane separation theorem for convex sets there exists a hyperplane through $x_0$ with a normal vector $u\in\mathbb{R}^d$ such that $\langle u,x_i \rangle \leq \langle u, x_0\rangle $ for all $i\in\{1,\hdots,d+1\}$, where $\langle \cdot , \cdot \rangle$ denotes the standard scalar product on $\R^d$. Define the set $R = \{x_0+r u: r\in[0,\infty)\}$. For any $y\in R$, $x_0$ is the closest point to $y$ out of $\{x_0,\hdots,x_{d+1}\}$, whence $R\subset N(x_0, \sum_{j=0}^{d+1} \delta_{x_j})$ and $N(x_0, \sum_{j=0}^{d+1} \delta_{x_j})$ is unbounded. This gives us a contradiction and, thus, proves part a).

Let $i\in\{1,\hdots,d+1\}$ and assume that  $N(x_i, \sum_{j=0}^{d+1} \delta_{x_j})$  is bounded. It follows from part a) that $x_i\in\operatorname{int}(\operatorname{conv}(\{x_0,\hdots,x_{i-1},x_{i+1},\hdots,x_{d+1}\}))$. On the other hand, again by part a), we have that $x_0\in\operatorname{int}(\operatorname{conv}(\{x_1,\hdots,x_{d+1}\}))$. This implies that 
$$
\operatorname{conv}(\{x_0,\hdots,x_{i-1},x_{i+1},\hdots,x_{d+1}\})
=\operatorname{conv}(\{x_0,\hdots,x_{d+1}\})
=\operatorname{conv}(\{x_1,\hdots, x_{d+1}\}),
$$
and, thus, either $x_i, x_0 \in\operatorname{conv}(\{x_1,\hdots,x_{i-1},x_{i+1},\hdots,x_{d+1}\})$ or $x_i=x_0$. This gives us a contradiction and concludes the proof of part b). 
\end{proof}
Finally, we are in position to prove the main result of this section.
\begin{proof}[Proof of Theorem \ref{main-Thm-circ}]
From Lemma \ref{calc-meas} and \eqref{M()} we deduce that $M_t(B)\rightarrow M(B)$ as $t\to \infty$ for all $B \in\mathcal{I}$. Then, by  Theorem \ref{real} it is sufficient to show  
	\begin{align*}
	&\underset{t\rightarrow\infty}{\lim}\,t \int_{W}\mathbb{E}\Big[\mathbf{1}\{s_t \mu(\mathbf{B}(x , C(x,\eta_{t} +\delta_x)))\in B \}
	\\
	&\hspace{2cm}\times \mathbf{1}\Big\{\sum_{y\in\eta_{t}\cap W} \delta_{s_t \mu(\mathbf{B}(y , C(y,\eta_{t} +\delta_x)))}(B)=m\Big\}\Big] f(x)dx - M(B)\mathbb{P}(\xi_t(B)=m) =0
	\end{align*}
	for $m\in\N_0$ and $B\in\mathcal{I}$. Put $\overline{u} = \sup(B)$ and let $\ell_t = 4\big(\frac{2\overline{u}}{\beta s_t k_d}\big)^{1/d}$. We write
	\begin{align}
	& t \int_{W}\mathbb{E}\Big[\mathbf{1}\big\{s_t \mu(\mathbf{B}(x , C(x,\eta_{t} +\delta_x)))\in B \big\}\mathbf{1}\Big\{ \sum_{y\in\eta_{t}\cap W} \delta_{s_t \mu(\mathbf{B}(y , C(y,\eta_{t} +\delta_x)))}(B)=m\Big\} \Big] f(x)dx 
	\notag
	\\
	& = t \int_{W} \mathbb{E}\Big[\mathbf{1}\big\{s_t \mu(\mathbf{B}(x , C(x,\eta_{t} +\delta_x))) \in B \big\} \mathbf{1}\big\{\eta_t(\mathbf{B}(x,  \ell_t))=d+1\big\}
	\notag
	\\
	& \qquad\qquad\qquad \times \mathbf{1}\Big\{\sum_{y\in\eta_{t}\cap W} \delta_{s_t \mu(\mathbf{B}(y , C(y,\eta_{t} +\delta_x)))}(B)= m \Big\} \Big] f(x)dx 
	\notag
	\\
	&\quad  + t \int_{W}\mathbb{E}\Big[\mathbf{1}\big\{s_t \mu(\mathbf{B}(x , C(x,\eta_{t} +\delta_x)))\in B \big\} \mathbf{1}\big\{\eta_t(\mathbf{B}(x,  \ell_t))>d+1\big\}
	\notag
	\\
	& \qquad\qquad\qquad \times \mathbf{1}\Big\{\sum_{y\in\eta_{t}\cap W} \delta_{s_t \mu(\mathbf{B}(y , C(y,\eta_{t} +\delta_x)))}(B)= m \Big\} \Big] f(x)dx 
	\notag
	\\
	& =:  A_t + R_t .
	\notag
	\end{align}
	By Lemma \ref{calc-meas}, we obtain $R_t \to 0$ as $t \to \infty$. Let us study $A_t$. From Lemma \ref{lem:bnd on g} it follows that there exists $t_0>0$ such that $\overline{u}/s_t\in T$ and $\ell_t \geq 4g(y,\overline{u}/s_t)$ for all $y\in W$ and $t>t_0$. Assume $t>t_0$. In case there are only $d+1$ points of $\eta_t$ in $\mathbf{B}(x,\ell_t)$, we deduce that $s_t \mu(\mathbf{B}(x , C(x,\eta_{t} +\delta_x)))\in B$ only if the $d+1$ points belong to $\mathbf{B}(x,2 g(x, \overline{u}/s_t))$. Then,  by $\ell_t \geq 4g(x,\overline{u}/s_t)$ we obtain 
	\begin{equation}\begin{split}\label{eq: fisrtstepA_t}
	A_t 
	& =  t \int_{W}\,\mathbb{E}\Big[\mathbf{1}\big\{ s_t \mu(\mathbf{B}(x , C(x,\eta_{t} +\delta_x)))\in B\big\}
	\\
	& \hspace{2cm} \times \mathbf{1}\big\{\eta_t(\mathbf{B}(x,\ell_t)\setminus \mathbf{B}(x,\ell_t/2))=0, \eta_t(\mathbf{B}(x,\ell_t/2))=d+1\big\}
	\\
	& \hspace{2cm} \times \mathbf{1}\Big\{\sum_{y\in\eta_{t}\cap W} \delta_{s_t \mu(\mathbf{B}(y , C(y,\eta_{t} +\delta_x)))}(B)=m \Big\} \Big]   f(x) dx . 
	\end{split}\end{equation}
	Furthermore, since  $\ell_t \geq 4g(y,\overline{u}/s_t)$ for all $y\in W$, we have  that
	$$
	 \mathbf{B}(y, 2g(y, \overline{u} /s_t))\cap \mathbf{B}(x,\ell_t/2) = \emptyset,  \quad y\in  \mathbf{B}(x,\ell_t)^c\cap W.
	$$
 Now the observation that
	$$
	s_t \mu(\mathbf{B}(y , C(y,\eta_{t} +\delta_x)))\in B
	\quad \text{if and only if} \quad
	s_t \mu(\mathbf{B}(y , C(y,(\eta_{t}+\delta_x)|_{\mathbf{B}(y,2g(y, \overline{u}/s_t))} )))\in B
	$$ 
	for $y\in\eta_t$ establishes that 
\begin{align*}
	A_t 
	& =  t \int_{W}\,\mathbb{E}\Big[\mathbf{1}\big\{ s_t \mu(\mathbf{B}(x , C(x,\eta_{t} +\delta_x)))\in B\big\}
	\\
	& \hspace{1.5cm} \times \mathbf{1}\big\{\eta_t(\mathbf{B}(x,\ell_t)\setminus \mathbf{B}(x,\ell_t/2))=0, \eta_t(\mathbf{B}(x,\ell_t/2))=d+1\big\}
	\\
	& \hspace{1.5cm} \times \mathbf{1}\Big\{\xi_t(\eta_t|_{\mathbf{B}(x,\ell_t)^c})(B) +\sum_{y\in\eta_{t}\cap \mathbf{B}(x,\ell_t/2) \cap W} \delta_{s_t \mu(\mathbf{B}(y , C(y,\eta_{t} +\delta_x)))}(B)   =m \Big\} \Big]   f(x) dx . 
\end{align*}
	Suppose that  $s_t \mu(\mathbf{B}(x , C(x,\eta_{t} +\delta_x)))\in B$ and that there are exactly $d+1$ points $y_1,\dots , y_{d+1}$ of $\eta_t$ in $\mathbf{B}(x,\ell_t/2)$ and $\eta_t\cap \mathbf{B}(x,\ell_t)\cap \mathbf{B}(x,\ell_t/2)^c=\emptyset$. 
	From Lemma \ref{lemmconvhull} it follows that $x\in\operatorname{int}(\operatorname{conv}(\{y_1,\hdots,y_{d+1}\}))$ and that the Voronoi cells $N(y_i, \eta_t|_{\mathbf{B}(x,\ell_t)} +\delta_x), i=1,\dots,d+1$, are unbounded. In particular, we have
	$$
	C(y_i, \eta_t + \delta_x)> \ell_t/4 > g(y_i, \overline{u}/s_t), \quad i=1,\dots, d+1.
	$$
	Together with the same arguments used to show \eqref{eq: fisrtstepA_t}, this implies that
	\begin{align*}
	A_t 
    & = t \int_{W}\,\mathbb{E}\big[\mathbf{1}\big\{ s_t \mu(\mathbf{B}(x , C(x,\eta_{t} +\delta_x)))\in B\big\} 
		\\
	& \hspace{2cm} \times \mathbf{1}\big\{\eta_t(\mathbf{B}(x,\ell_t)\setminus \mathbf{B}(x,\ell_t/2))=0, \eta_t(\mathbf{B}(x,\ell_t/2))=d+1\big\}
	\\
	& \hspace{2cm} \times \mathbf{1}\big\{ \xi_t(\eta_t|_{\mathbf{B}(x,\ell_t)^c} )(B) = m \big\} \big]  f(x) dx 
	\\
	& = t \int_{W}\,\mathbb{E}\big[\mathbf{1}\big\{ s_t \mu(\mathbf{B}(x , C(x,\eta_{t} +\delta_x)))\in B\big\} \mathbf{1}\big\{\eta_t(\mathbf{B}(x,\ell_t))=d+1\big\}\big] 
	\\
	& \hspace{2cm} \times \mathbb{P}\big( \xi_t(\eta_t|_{\mathbf{B}(x,\ell_t)^c}) (B) = m \big)   f(x) dx  .
	\end{align*}
Furthermore, we obtain
	\begin{align*}
	\big\vert\mathbb{P}\big( \xi_t(\eta_t|_{\mathbf{B}(x,\ell_t)^c} ) (B) = m \big)- \mathbb{P}\big( \xi_t(B) = m \big) \big\vert  
	\leq \mathbb{P}(\eta_t(\mathbf{B}(x,\ell_t))>0)< t \mu(\mathbf{B}(x,\ell_t))
	\end{align*}
	for any $x\in W$, where we used the Markov inequality in the last step. Combining the previous formulas leads to
	\begin{align*}
	& \big| A_t - M(B)\mathbb{P}(\xi_t(B)=m) \big| \\
	& \leq |M_t(B)-M(B)| \\
	& \quad + t \int_{W}\,\mathbb{E}\big[\mathbf{1}\big\{ s_t \mu(\mathbf{B}(x , C(x,\eta_{t} +\delta_x)))\in B\big\} \mathbf{1}\big\{\eta_t(\mathbf{B}(x,\ell_t))>d+1\big\}\big] f(x) dx \\
	& \quad + t \int_{W}\,\mathbb{E}\big[\mathbf{1}\big\{ s_t \mu(\mathbf{B}(x , C(x,\eta_{t} +\delta_x)))\in B\big\}\mathbf{1}\big\{\eta_t(\mathbf{B}(x,\ell_t))=d+1\big\}\big]
	\\
	& \hspace{2cm}
	\times \big| \mathbb{P}\big( \xi_t(\eta_t|_{\mathbf{B}(x,\ell_t)^c}) (B) = m \big) - \mathbb{P}(\xi_t(B)=m) \big|  f(x) dx \\
	& \leq |M_t(B)-M(B)| + t \int_{W}\,  \mathbb{P}\big(\eta_t(\mathbf{B}(x,\ell_t))>d+1\big) f(x) dx + \widehat{M}_t([0,\overline u]) \sup_{x\in W}  t \mu(\mathbf{B}(x,\ell_t)). 
	\end{align*}
	It follows from Lemma \ref{calc-meas} that, as $t\to\infty$, $\widehat{M}_t([0,\overline u])\to M([0,\overline u])$, $M_t(B)\to M(B)$ and the integral on the right-hand side vanishes. Without loss of generality we may assume $\ell_t\leq 1$, and thus the  continuity of $f$ on $W+\overline{\mathbf{B}(0,1)}$  implies that 
	$$
	t \mu(\mathbf{B}(x,\ell_t)) \leq k_d \max_{z\in W + \overline{\mathbf{B}(0,1) }} f(z) t \ell_t^d
	$$
for all $x\in W$.	Now $\ell_t = 4\big(\frac{2 \overline{u}}{\beta s_t k_d}\big)^{1/d}$ and $s_t = \alpha_2 t^{(d+2)/(d+1)}$ yield that the right-hand side vanishes as $t\to\infty$. Thus, we obtain
$$
\lim_{t\to\infty} A_t - M(B)\mathbb{P}(\xi_t(B)=m) = 0,
$$
which together with $R_t\to0$ as $t\to\infty$ concludes the proof.
\end{proof}

\begin{proof}[Proof of Theorem \ref{min circradii}]
Let  $\gamma$ be a Poisson  process on $\R^d\times [0,\infty)$ with the restriction of the Lebesgue measure as intensity measure. Let $\mu_1$ and $\mu_2$ denote the absolutely continuous measures with densities $f_1$ and $f_2$, respectively. 
Then, \cite[Corollary 5.9 and Proposition 6.16]{MR3791470} imply that
$$
\varrho^{(1)}_t 
= \sum_{(x,y)\in \gamma} \mathbf{1} \{ y\leq t f_1(x)\} \delta_{x},
\quad
\varrho^{(2)}_t  = \sum_{(x,y)\in \gamma} \mathbf{1} \{ y\leq t f_2(x)\} \delta_{x}
$$
and
$$
\varrho_t 
= \sum_{(x,y)\in \gamma} \mathbf{1} \{ y\leq t \phi(x)\} \delta_{x}
$$
are Poisson  processes on $\R^d$ with intensity measures $t \mu_1, t\mu_2$ and  $t \vartheta$, respectively. They satisfy
$$
\varrho^{(1)}_t(A) \leq \varrho_t(A) \leq \varrho^{(2)}_t (A) \,\,\, \text{a.s.}
\quad 
\text{ and }
\quad 
\varrho_t  \overset{d}{=} \zeta_t, \qquad A\subset \R^d, \, t>0.
$$
Therefore for any $v\geq0$, we obtain
\begin{align}\label{eq:ineq bnd mu1}
\mathbb{P}\big(\min_{x\in\zeta_t\cap W} \mu_1(\mathbf{B}(x, C(x,\zeta_t))) > v \big)
\leq \mathbb{P}\big(\min_{x\in\varrho_t^{(1)}\cap W} \mu_1(\mathbf{B}(x, C(x,\varrho_t^{(1)}))) >v \big)
\end{align}
and similarly 
\begin{align}\label{eq:ineq bnd mu2}
\mathbb{P}\big(\min_{x\in\zeta_t\cap W} \mu_2(\mathbf{B}(x, C(x,\zeta_t))) > v \big)
\geq \mathbb{P}\big(\min_{x\in\varrho_t^{(2)}\cap W} \mu_2(\mathbf{B}(x, C(x,\varrho_t^{(2)}))) > v \big).
\end{align}
From Theorem \ref{main-Thm-circ}, it follows for $j=1,2$ and $\nu(t)= u(\alpha_2t^{(d+2)/(d+1)})^{-1}$ with  $u\geq 0$, that
$$
\lim_{t\rightarrow\infty} \mathbb{P}\big(\min_{x\in\varrho_t^{(j)}\cap W} \mu_j(\mathbf{B}(x, C(x,\varrho_t^{(j)}))) >\nu(t) \big) = e^{-\mu_j(W)u^{d+1}}.
$$
If $s\phi \leq f_1 \leq \phi$ for some $s\in (0,1]$, combining \eqref{eq:ineq bnd mu1}, the previous limit with $j=1$,  and the inequality
$$
\mathbb{P}\big(\min_{x\in\zeta_t\cap W} s\vartheta(\mathbf{B}(x, C(x,\zeta_t))) >\nu(t) \big)\leq \mathbb{P}\big(\min_{x\in\zeta_t\cap W} \mu_1(\mathbf{B}(x, C(x,\zeta_t))) >\nu(t) \big)
$$
implies that 
$$
\limsup_{t\rightarrow\infty} \mathbb{P}\big(\min_{x\in\zeta_t\cap W} s \vartheta(\mathbf{B}(x, C(x,\zeta_t))) >\nu(t) \big) \leq  e^{-\mu_1(W)u^{d+1}}.
$$
Then, $s\vartheta(W)\leq \mu_1(W)$  concludes the proof of  $(i)$. Analogously, if $\phi \leq f_2 \leq r \phi $ for some $r\geq 1$, combining \eqref{eq:ineq bnd mu2}, the limit above with $j=2$, the inequality
$$
\mathbb{P}\big(\min_{x\in\zeta_t\cap W} r \vartheta(\mathbf{B}(x, C(x,\zeta_t)))>\nu(t) \big)\geq \mathbb{P}\big(\min_{x\in\zeta_t\cap W} \mu_2(\mathbf{B}(x, C(x,\zeta_t))) >\nu(t) \big)
$$
and $ \mu_2(W)\leq r\vartheta(W)$ for $u\geq 0$ shows  $(ii)$.
\end{proof}

\bibstyle{abbrv}

\bibliography{bibliography}

\end{document}